\documentclass[12pt]{article}

\usepackage{amsfonts}
\usepackage{amssymb}
\usepackage{amsmath}
\usepackage{enumitem}
\usepackage{float}
\usepackage{xspace}
\usepackage{geometry}
\usepackage{centernot}
\usepackage{cancel}

\usepackage{ dsfont }

\usepackage[pdftex]{graphicx}

\usepackage{graphicx}

\usepackage{amsthm}
\usepackage{thmtools}

\newtheorem{lemma}{Lemma}[section]
\newtheorem{proposition}[lemma]{Proposition}
\newtheorem{corollary}[lemma]{Corollary}
\newtheorem{theorem}[lemma]{Theorem}
\theoremstyle{definition}
\newtheorem{definition}[lemma]{Definition}
\newtheorem{remark}[lemma]{Remark}
\newtheorem{example}[lemma]{Example}

\theoremstyle{definition}

\theoremstyle{definition}

\newtheoremstyle{mythmstyle}%
    {}%
    {}%
    {\it}%
    {}%
    {\bf}%
    {}%
    { }%
    {\thmname{#1}\thmnumber{ #2}%
    \thmnote{: #3\addcontentsline{toc}{subsubsection}{{\it#1}: #3}}. }
\theoremstyle{mythmstyle}

\usepackage[capitalize]{cleveref}

\Crefname{theorem}{Theorem}{Theorems}
\Crefname{lemma}{Lemma}{Lemmas}
\Crefname{definition}{Definition}{Definitions}
\Crefname{proposition}{Proposition}{Propositions}
\Crefname{corollary}{Corollary}{Corollaries}

\usepackage{tikz}
\usetikzlibrary{intersections,through,graphs,graphs.standard,quotes,arrows,arrows.meta,decorations.markings,fit,positioning,calc}

\DeclareMathOperator{\ker/}{ker}

\DeclareMathOperator{\im/}{Im}
\DeclareMathOperator{\dom/}{Dom}
\DeclareMathOperator{\rk/}{rk}
\DeclareMathOperator{\aut/}{Aut}

\DeclareMathOperator{\pat/}{PAut}

\DeclareMathOperator{\sub/}{Sub}

\def\Tn/{\mathcal{T}_n}
\def\In/{\mathcal{I}_n}
\def\On/{\mathcal{O}_n}
\def\Sn/{\mathcal{S}_n}
\def\An/{\mathbf{A}_n}
\def\lcs/{$\mathcal{LC}(S)$}
\def\rcs/{$\mathcal{RC}(S)$}
\def\gwt/{G\wr\mathcal{T}_n}
\def\gwi/{G\wr\mathcal{I}_n}
\def\gwip/{G\wr\mathcal{I}_{n+1}}
\def\gws/{G\wr\mathcal{S}_n}
\def\gwsk/{G\wr\mathcal{S}_k}
\def\bA/{\mathbf{A}}
\def\lA/{\mathcal{L}(\bA/)}
\def\iqm/{\mathbf{Q}_m}

\newcommand\textequal{%
 \rule[.4ex]{4pt}{0.4pt}\llap{\rule[.7ex]{4pt}{0.4pt}}}

\usepackage{titlesec}
\titleformat{\section}
  {\normalfont\large\scshape}{\thesection}{1em}{}
  \titleformat{\subsection}
  {\normalfont\normalsize\scshape}{\thesubsection}{1em}{}

\providecommand{\abstract}[1]
{
  \small	
  \textbf{{Abstract}} #1
}

\usepackage{authblk}
\author{Matthew D G K Brookes \thanks{Email: {\tt mdkb500@york.ac.uk}}}
\affil{{\small \em Department of Mathematics, University of York,} \\ \small{\em York, YO10 5DD}}
 
\def\keywords{\small \vspace{.5em}
{\textbf{Keywords:}\,\relax%
}}

\usepackage{titling}

\predate{}
\postdate{}
\date{}

\pretitle{\begin{center}\huge\scshape}
\posttitle{\end{center}}

\begin{document}

\title{Congruences on the Partial Automorphism Monoid of a Free Group Action}

\maketitle

\begin{abstract}
We study congruences on the partial automorphism monoid of a finite rank free group action.  We give a decomposition of a congruence on this monoid into a Rees congruence, a congruence on a Brandt semigroup and an idempotent separating congruence.  The constituent parts are further described in terms of subgroups of direct and semidirect products of groups.  We utilize this description to demonstrate how the number of congruences on the partial automorphism monoid depends on the group and on the rank of the action.
\end{abstract}

\keywords{Free group action, Partial automorphism monoid, Congruences, Subgroups of direct products}

\let\thefootnote\relax\footnote{AMS Mathematics Subject Classification 2010: 20M20, 16W22, 20M18}

\section{Introduction}

The study of congruences is acknowledged as fundamental to understanding the structure of semigroups. Inverse semigroups are arguably the most studied class of semigroups, and there are well trodden paths for the hardy semigroup theorist to follow in order to get a hold of their congruences \cite{congs inv 1}, \cite{congs inv 2}.  Of central importance in the field of inverse semigroup theory is the symmetric inverse monoid $\mathcal{I}_X,$ which plays the same role as the symmetric group does within group theory. Congruences on the symmetric inverse monoid are well understood \cite{Liber congs I_n}. 

Monoids and semigroups similar to the symmetric inverse monoid in derivation or structure are valuable and interesting objects of study.  Recent forays into the study of congruences on ``transformation-like'' semigroups include \cite{congs diagram} and \cite{congs direct prods}. Natural generalizations of $\mathcal{I}_X$ and $\mathcal{T}_X$ arise from the partial automorphism monoids and endomorphism monoids of independence algebras, a concept introduced as $v^*$-algebras in \cite{v* alg} and formulated in its modern form in \cite{Gould ind alg} and \cite{Fountain ind alg}. Of course, we can regard a set $X$ as a universal algebra with no basic operations; viewed in this way it is an independence algebra. 

The class of independence algebras generalizes notions of linear independence and spanning sets and includes the classes of sets, vector spaces and free group actions.  Endomorphism monoids of independence algebras are studied in \cite{Gould ind alg}. In this paper we focus on the partial automorphism monoid of a finite rank free group action. Such monoids will have a wreath product like structure and so we denote then by $G\wr\mathcal{I}_n,$ where $G$ is the group in question and $n$ the rank. 

We build on general results concerning congruences on partial automorphism monoids of independence algebras due to Lima \cite{Lima thesis} and describe congruences on $\gwi/$ in terms of normal subgroups of $G^i$ for $1\leq i\leq n$ and a subgroup of $G\wr \mathcal{S}_m$ for some $1\leq m\leq n.$  This leads us to consider permutation invariant normal subgroups of $G^i,$ and we shall see that these subgroups can be described in terms of three normal subgroups of $G$ and a homomorphism.

One of the most basic questions that can be asked about the set of congruences on a semigroup is how large is this set? Our techniques permit us to determine bounds for the number of congruences on $\gwi/$ depending on both $n$ and $G.$

The paper is set out as follows.  In Section $2$ we introduce the partial automorphism monoid for a rank $n$ free group action and show that it is isomorphic to the partial wreath product $\gwi/$.  Section $3$ comprises a discussion of congruences on $\gwi/$ and gives a unique decomposition for a congruence in terms of a Rees congruence, a set of subgroups of $G^i$ for $1\leq i\leq n$ and a subgroup of $G\wr\mathcal{S}_m$ for some $1\leq m\leq n.$  Moreover we demonstrate that this decomposition is compatible in a natural way with the usual inclusion ordering on congruences.  In Section $4$ we analyze the subgroups arising in Section $3$ and describe permutation invariant subgroups of $G^i$ in terms of subgroups of $G$ and a homomorphism from a subgroup of $G$ to a quotient of that same subgroup. By appeal to Usenko's description of subgroups of semidirect products we classify normal subgroups of $G\wr\mathcal{S}_m.$  Finally, in Section $5$ we describe the asymptotic behavior of the number of congruences as $n$ increases and determine how this growth depends on the group in question. 

\section{Preliminaries}

Throughout $G$ is a group and $S$ an inverse semigroup, $E(S)$ is the set of idempotents of $S$ where $e\in S$ is idempotent if $e^2=e$. For a group $G$ we define $G^0$ to be the group with a zero adjoined, so $G^0=G\cup \{0\}$ with multiplication extended by declaring $g0=0=0g$ for all $g\in G^0.$  By $\Sn/$ we refer to the symmetric group and by $\In/$ the symmetric inverse monoid of degree $n.$  For $a\in \In/$ we define the rank of $a$ to be $\rk/(a)=|\im/(a)|.$ As usual, $\mathcal{H,R,L,D,J}$ denote Green's relations, and for $a\in S$ and $K\in {H,R,L,D,J}$ we write $K_a$ for the $\mathcal{K}$-class of $a$ (we use $\mathcal{K}(S)$ and $K(S)_a)$ when $S$ is ambiguous). We shall want to refer to the lattice of congruences on $S$ and we denote this by $\mathfrak{C}(S),$ in the case of a group we write $\mathfrak{N}(G)$ for the lattice of normal subgroups of $G.$  For a congruence $\rho$ on $S$ let $a\rho$ be the $\rho$-class of $a.$  We denote by $[n]$ the set $\{1,2,\dots,n\}.$

\subsection{The partial automorphism monoid of a free group action}

A {\em group action} or {\em $G$-act} $\mathbf{A}$ is a non empty set $A$ together with a function $\Phi: G\times A\rightarrow A$ where $(g,a)\mapsto g\cdot a,$ such that $(gh)\cdot a = g\cdot (h\cdot a)$ and $1\cdot a=a$ for all $g,h\in G$ and $a\in A.$  To simplify notation we write $ga$ for $g\cdot a.$  A subset $B\subseteq A$ is a {\em subact} of $\mathbf{A}$ if $B$ is also a $G$-act under the restriction of the action to $B.$ For $G$-acts $\mathbf{A}$ and $\mathbf{B}$ a function $\gamma\colon A\rightarrow B$ is a {\em $G$-act homomorphism} if $g(a\gamma)=(ga)\gamma$ for all $g\in G$ and $a\in A.$  The {\em free $G$-act over a set $X$} is written as $\mathbf{A}_X=G\times X,$ and $g(h,x)=(gh,x).$ We shall usually drop the brackets and write $gx$ for $(g,x),$ we identify $(1,x)=1x$ with $x$ and use $Gx$ to denote $\{gx\mid g\in G\}.$ When $X=\{x_i\mid i\in [n]\}$ then we write $\mathbf{A}_X=\mathbf{A}_n.$  The {\em rank} of a free $G$-act is the cardinality of $X,$ so $\mathbf{A}_n$ is the free $G$-act of rank $n.$

For an algebra (in the sense of universal algebra) a subalgebra is a subset which is an algebra of the same type upon the restriction of the operations; we denote the subalgebras of $\mathbf{A}$ by $\sub/(\mathbf{A}).$ The {\em partial automorphism monoid}, $\pat/(\mathbf{A})$ for an algebra $\mathbf{A}$ is the set of isomorphisms between two (not necessarily distinct) subalgebras  
$$\pat/(\mathbf{A})=\{f\colon \mathbf{B}\rightarrow \mathbf{A}\mid \mathbf{B}\subseteq \mathbf{A} \text{ a subalgebra},\ f \text{ an injective homomorphism}\},$$
under composition of partial functions.  Associated to $\pat/(\mathbf{A})$ are the domain function $\dom/:\pat/(\mathbf{A})\rightarrow \sub/(\mathbf{A})$ and image function $\im/:\pat/(\mathbf{A})\rightarrow \sub/(\mathbf{A}).$  Concretely, the composition of $a,b\in \pat/(\mathbf{A})$ has $\dom/(ab)=(\im/(a)\cap \dom/(b))a^{-1}$ and $\im/(ab)=(\im/(a)\cap \dom/(b))b$, with $a^{-1}$ the inverse of $a$ as a partial function, and $x(ab)=(xa)b$ for all $x\in \dom/ ab.$  With this operation the set of partial automorphisms is an inverse monoid.  This has a zero which is the empty map, the group of units is the automorphism group of $\mathbf{A}$ and the semilattice of idempotents is actually a lattice, and is isomorphic to the lattice of subalgebras of $\mathbf{A}.$

\begin{lemma}\label{subacts}
Let $\mathbf{A}_X$ be a free $G$-act. If $Y\subseteq X$ then $\mathbf{A}_Y$ is a subact of $\mathbf{A}_X.$  Conversely if $\mathbf{B}\subseteq \mathbf{A}_X$ is a subact then there is $Y\subseteq X$ such that $\mathbf{B}=\mathbf{A}_Y.$ Consequently if $\mathbf{B}\subseteq\mathbf{A}$ a subact then $\mathbf{B}$ is a free $G$-act.
\end{lemma}

\begin{proof}
This is immediate noting that for $x\in X$ if there is $g\in G$ such that $gx\in B$ then certainly $Gx\subseteq B.$
\end{proof}

Lemma \ref{subacts} gives that for the $n$-rank free $G$-act $\mathbf{A}_n$ the set of subacts is isomorphic to the set of subsets of $[n].$ Moreover, for $\theta\in \pat/(\mathbf{A}_n),$ if $x_i\theta = kx_j$ it is clear that $\theta|_{Gx_i}$ is a $G$-act isomorphism from $Gx_i$ to $Gx_j.$  Thus a partial automorphism $\theta\in \pat/\mathbf{A}_n$ defines an element $a_\theta\in \mathcal{I}_n$ where $i\in \dom/a_\theta$ if and only if $x_i\in \dom/\theta,$ and $ia_\theta=j$ where $x_i\theta\in Gx_j.$  Furthermore it is clear that the map $\theta \mapsto a_\theta$ defines a homomorphism $\pat/(\mathbf{A}_n)\rightarrow \In/.$ 

For a set $X$ we write $G^{PX}=\{f\colon Y\rightarrow G\mid  Y\subseteq X\}$ for the set of partial functions from $X$ to $G.$  We define the product of $f,g\in G^{PX}$ by: 
$$\dom/(fg)=\dom/(f)\cap\dom/(g),\quad x(fg)=(xf)(xg) \text{ for } x\in \dom/(fg).$$

Given $\theta\in \pat/(\mathbf{A}_n)$ we define $f_\theta$ as the group label of the restriction of $\theta$ to the set $\dom/(\theta)\cap \{x_i\mid i\in [n]\}$ (so if $x_i\theta=gx_j$ then $f_\theta=g$). Then $f_\theta\in G^{P[n]}$ and we notice that $\theta$ is fully characterized by $f_\theta$ and $a_\theta$ as this defines $x\theta$ for each $x$ with $Gx\subseteq \dom/(\theta)$ and we can then uniquely extend $\theta$ to the rest of $\dom/(\theta)$ by the freeness of $\mathbf{A}_n.$

This leads us to consider a partial wreath product as defined in \cite{meldrum1995wreath}.  For $a\in \In/$ and $f\in G^{P[n]}$ define $f^a$ as:
$$\dom/(f^a)=\{i\in \dom/(a)\mid ia\in\dom/(f)\},\quad i(f^a)=f(ia).$$
Then we define the {\em partial wreath product} $\gwi/$ to be:
$$\gwi/=\{(f;a)\in G^{P[n]}\times \In/\mid \dom/(f)=\dom/(a)\},$$
with multiplication
$$(f;a)(g;b)=(fg^a;ab).$$
The proof of the following theorem is essentially folklore, we present an outline here for completeness. 

\begin{theorem}
The function
$$\Phi: \pat/\mathbf{A}_n\rightarrow \gwi/;\ \ \theta\mapsto (f_\theta;a_\theta)$$
is an isomorphism.
\end{theorem}

\begin{proof}
As we have noted a partial automorphism $\theta$ is determined by $f_\theta$ and $a_\theta,$ thus $\Phi$ is injective. It is also straightforward that given $f\in G^{P[n]}$ and $a\in \In/$ with $\dom/(a)=\dom/(f)$ the function $\theta: x_i\mapsto (if)x_{ia}$ extends uniquely to a partial automorphism with $f=f_\theta$ and $a=a_\theta$ so $\Phi$ is surjective.  Thus it remains to show that $\Phi$ is a homomorphism.

To this end suppose that $\theta,\gamma\in \pat/(\mathbf{A}_n).$ Then as the function $\theta\mapsto a_\theta$ is a homomorphism we have that $a_{\theta\gamma}=a_\theta a_\gamma$ and it is clear that $\dom/(f_{\theta\gamma})=\dom/(a_{\theta\gamma}).$  Also for $i\in \dom/(a_{\theta\gamma})$ we have $i f_{\theta\gamma}$ is the group label of $x_i(\theta\gamma)= (x_i\theta)\gamma= (f_\theta x_{ia_\theta})\gamma,$ and the group label of $x_{ia_\theta}\gamma$ is $(f_\gamma)a.$  Thus $i f_{\theta\gamma}= f_\theta f_\gamma^{a_\theta}$ and so $\Phi$ is a homomorphism. 
\end{proof}

While this is the morally correct way to construct $\gwi/,$ and justifies us referring to it as a partial wreath product is it generally unwieldy, and sends one down endless rabbit-holes of notational difficulties.  Fortunately it is possible to pin down a less aesthetically pleasing but more ``user-friendly'' version.

To establish this, let $a\in \In/$, and $g\in (G^0)^n.$  Write $g_a=(g_{1a},\dots,g_{na})$ where we take $g_{ia}=0$ if $ia$ is undefined. In particular we write $1_e$ for $(1,1\dots,1)_e$. We define 
$$S=\{ (g,a)\in (G^0)^n\times \In/\ |\ g_i\neq 0 \iff i\in \dom/(a)\},$$
with multiplication 
$$(g;a)(h;b)=(g_1,\dots,g_n;a)(h_1,\dots,h_n;b)=(g_1h_{1a},\dots,g_nh_{na};ab)=(gh_a;ab).$$ 
It is elementary that $S\cong \gwi/$ via the map $(g;a)\mapsto (f;a)$ where $\dom/(f)=\{i\in [n]\mid g_i\neq 0\}$ and $if=g_i.$  From this point forward we shall refer to this second formulation when we write $\gwi/.$

As $\pat/(\mathbf{A}_n)$ is inverse we have that $\gwi/$ is an inverse semigroup, and we observe that $$(g;a)^{-1}=(g_1,\dots,g_n;a)^{-1}=(g_{1a^{-1}}^{-1},\dots,g_{na^{-1}}^{-1};a^{-1})=(g_{a^{-1}}^{-1};a^{-1})$$
where we write $0^{-1}=0.$  Also notice that the condition $g_i\neq 0$ if and only if $i\in \dom/(a)$ for $(g;a)$ to be in $\gwi/$ can be reformulated as $g=g_{aa^{-1}}$ and $g\neq g_e$ for all $e\in E(\In/)$ with $e\leq aa^{-1}.$

We can view this semigroup pictorially. As is fairly usual we consider $a\in\In/$ as a graph with two rows each of $n$ vertices - indexed as $1,2,\dots,n$ for the upper row and $1^\prime,2^\prime,\dots,n^\prime$ for the lower row - with edges $(i,j^\prime)$ if $ia=j.$  We then consider $(g,a)\in\gwi/$ as the graph of $a$ with the top row labeled with elements $g\in G^0.$ We compose the graphs as elements of $\In/,$ and then ``slide'' the labels up adjacent edges. For an example refer to Figure \ref{gwimult}.

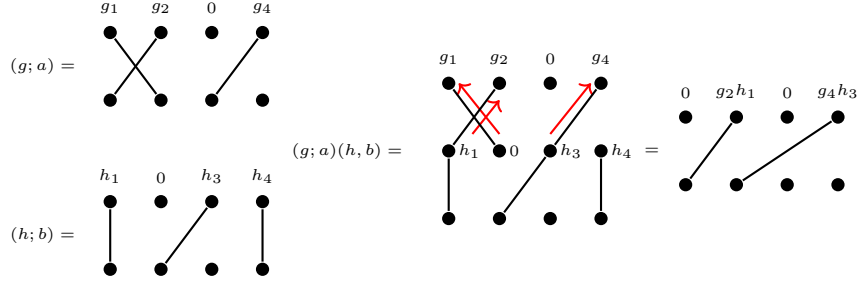
\begin{figure}[!h]
\centering
\begin{tikzpicture}[scale=.45,-,auto,thick,main node/.style={circle,fill=none,inner sep=0pt,minimum size=2pt},normal node/.style={circle,fill=black,inner sep=0pt,minimum size=5pt},this node/.style={circle,fill=grey,inner sep=0pt,minimum size=5pt}]

\node[main node] (a0) at (-6,6){\tiny  $(g;a) =$};
\node[main node] (b0) at (-6,1){\tiny  $(h;b) =$};
\node[main node] (c0) at (3,3.5){\tiny $(g;a)(h,b) =$};
\node[main node] (d0) at (12,3.5){\tiny $=$};

\node[normal node,label=above:{\tiny  $g_1$}] (a1) at (0-4,7){};
\node[normal node,label=above:{\tiny  $g_2$}] (a2) at (1.5-4,7){};
\node[normal node,label=above:{\tiny  $0$}] (a3) at (3-4,7){};
\node[normal node,label=above:{\tiny  $g_4$}] (a4) at (4.5-4,7){};

\node[normal node] (a5) at (0-4,5){};
\node[normal node] (a6) at (1.5-4,5){};
\node[normal node] (a7) at (3-4,5){};
\node[normal node] (a8) at (4.5-4,5){};

\node[normal node,label=above:{\tiny  $h_1$}] (b1) at (0-4,2){};
\node[normal node,label=above:{\tiny  $0$}] (b2) at (1.5-4,2){};
\node[normal node,label=above:{\tiny  $h_3$}] (b3) at (3-4,2){};
\node[normal node,label=above:{\tiny  $h_4$}] (b4) at (4.5-4,2){};

\node[normal node] (b5) at (0-4,0){};
\node[normal node] (b6) at (1.5-4,0){};
\node[normal node] (b7) at (3-4,0){};
\node[normal node] (b8) at (4.5-4,0){};

\draw (a1) -- (a6);
\draw (a2) -- (a5);
\draw (a4) -- (a7);

\draw (b1) -- (b5);
\draw (b3) -- (b6);
\draw (b4) -- (b8);

\node[normal node,label=above:{\tiny  $g_1$}] (c1) at (6,5.5){};
\node[normal node,label=above:{\tiny  $g_2$}] (c2) at (7.5,5.5){};
\node[normal node,label=above:{\tiny  $0$}] (c3) at (9,5.5){};
\node[normal node,label=above:{\tiny  $g_4$}] (c4) at (10.5,5.5){};

\node[normal node,label={[xshift=-0.12cm]right:{\tiny  $h_1$}}] (c5) at (6,3.5){};
\node[normal node,label={[xshift=-0.12cm]right:{\tiny  $0$}}] (c6) at (7.5,3.5){};
\node[normal node,label={[xshift=-0.12cm]right:{\tiny  $h_3$}}] (c7) at (9,3.5){};
\node[normal node,label={[xshift=-0.12cm]right:{\tiny  $h_4$}}] (c8) at (10.5,3.5){};

\draw[->,red] (7.5,4) -- (6.3,5.5);
\draw[->,red] (6.7,4) -- (7.5,5);
\draw[->,red] (9,4) -- (10.2,5.5);

\node[normal node] (d5) at (6,1.5){};
\node[normal node] (d6) at (7.5,1.5){};
\node[normal node] (d7) at (9,1.5){};
\node[normal node] (d8) at (10.5,1.5){};

\draw (c1) -- (c6);
\draw (c2) -- (c5);
\draw (c4) -- (c7);

\draw (c5) -- (d5);
\draw (c7) -- (d6);
\draw (c8) -- (d8);

\node[normal node,label=above:{\tiny  $0$}] (e1) at (13,4.5){};
\node[normal node,label=above:{\tiny  $g_2h_1$}] (e2) at (14.5,4.5){};
\node[normal node,label=above:{\tiny  $0$}] (e3) at (16,4.5){};
\node[normal node,label=above:{\tiny  $g_4h_3$}] (e4) at (17.5,4.5){};

\node[normal node] (e5) at (13,2.5){};
\node[normal node] (e6) at (14.5,2.5){};
\node[normal node] (e7) at (16,2.5){};
\node[normal node] (e8) at (17.5,2.5){};

\draw (e2) -- (e5);
\draw (e4) -- (e6);

\end{tikzpicture}
\caption{\small Multiplication in $\gwi/$}
  \label{gwimult}
\end{figure}

\bigskip

We next give some straightforward initial results about $\gwi/.$  By Lemma \ref{subacts} the subacts of $\mathbf{A}_n$ are in bijective correspondence with subsets of $[n];$ further as previously remarked the set of subacts is also in bijection with semilattice of idempotents of $\gwi/$.

\begin{corollary}
The idempotents $E(\gwi/)$ are precisely the elements of the form 
$$(1_e;e)$$
where $e\in \In/.$  Consequently, $E(\gwi/)$ forms a lattice isomorphic to the subsets of $[n]$ which is isomorphic to the lattice of idempotents $E(\In/).$ 
\end{corollary}
\begin{proof}
Clearly $(1_e;e)\in E(\gwi/)$ for any $e\in E(\In/).$ Conversely, suppose that $(g;a)(g;a)=(g;a).$  Then certainly $a^2=a$, hence $a=e\in E(\In/).$  Then we note that $g_a=g_{e}=g,$ and $g^2=g$; whence $g$ is an idempotent in $(G^0)^n,$ with $g_i=0$ exactly when $i\not\in \dom/ e,$ so $g=1_e.$  
\end{proof}

Noting that for $(g;a)\in \gwi/$
$$ (g^{-1};aa^{-1})(g;a)=(1_{aa^{-1}};a)=(g;a)(g_{a^{-1}}^{-1};a^{-1}a),$$
it is clear that the Green's relations for $\gwi/$ are induced by those for $\In/.$

\begin{lemma}
Let $\mathcal{K}\in \{\mathcal{H,L,R,D,J}\}$ be a Green's relation.  Then 
$$(g;a)\ \mathcal{K}^{\gwi/}\ (h;b) \iff a\ \mathcal{K}^{\In/}\ b.$$
\end{lemma}

As a consequence two-sided ideals of $\gwi/$ are inherited from ideals in $\In/$; for each $0\leq m\leq n$ the set 
$$I_m=\{(g;a)\in \gwi/\mid \rk/(a)\leq m\}$$
is an ideal of $\gwi/,$ and these are all the ideals of $\gwi/.$

\subsection{Congruences on Inverse Semigroups}

\begin{definition}
If $S$ is an inverse semigroup with semilattice of idempotents $E$ then the centralizer of $E$ is
$$E\zeta=\{ a\in S\ |\ \forall e\in E,\ ea=ae\}.$$ 
\end{definition}

For $\gwi/$ we know that the idempotents are elements of the form $(1_e;e)$ for some idempotent $e\in \In/.$  An elementary calculation gives that for $\gwi/$ the centralizer of the idempotents is the set
$$E\zeta= \{ (g;e)\in\gwi/\ |\  e\in E(\In/)\}.$$

An important family of congruences on any inverse semigroup are the idempotent separating congruences, that is, those that have at most one idempotent in each equivalence class, or equivalently those contained in Green's $\mathcal{H}$ relation \cite{Munn idem sep}. We write $\mathfrak{C}_{IS}(S)$ for the set of idempotent separating congruences, and recall that these form a sublattice of the $\mathfrak{C}(S)$ so in particular there are maximum and minimum such congruences: the minimum is the trivial relation which we write as $\iota,$ and the maximum we write as $\mu.$  This is a well studied class of congruences (see, for example, \cite{congs inv 2}) and for inverse semigroups such congruences are determined by their kernel, where by the kernel of a congruence $\rho$ we mean 
$$\ker/(\rho)=\{a\mid \exists e\in E(S) \text{ with } a\ \rho\ e\}.$$

\begin{definition}
Let $S$ be an inverse semigroup and $T\subseteq S$ a subsemigroup.  Say $T$ is {\em full} if $E\subseteq T,$ and say $T$ is {\em self conjugate} if for each $a\in S$ we have $aTa^{-1}\subseteq T.$  A subsemigroup $T$ is {\em normal} if $T$ is full, self conjugate, and inverse.
\end{definition}

\begin{theorem}[{\cite[Proposition~5.14]{congs inv 2}}]\label{prelim idem sep thm}
Let $S$ be an inverse semigroup.  The lattice of idempotent separating congruences on $S$ is isomorphic to the lattice of normal subsemigroups of $S$ contained in $E\zeta.$  
The following maps are mutually inverse lattice isomorphisms:
\begin{gather*}
T\mapsto \chi_T=\{ (a,b)\ |\ a^{-1}a=b^{-1}b,\ ab^{-1}\in T\},\\
\chi\mapsto \ker/(\chi) = \bigcup_{e\in E} e\rho.
\end{gather*}
\end{theorem}

The problem of describing the lattice of idempotent separating congruences on $S$ then becomes that of describing the lattice of normal subsemigroups of $S$ contained in $E\zeta.$  Lattices of subsemigroups of inverse semigroups are a well studied topic, see, for example, \cite{jones subsemis}. 

\section{The Congruence Decomposition}

Congruences on $\In/$ are well understood. Choose $1\leq k\leq n$ and $N\trianglelefteq \mathcal{S}_k.$  For each $a,b$ with $\rk/(a)=k=\rk/(b)$ and $a\ \mathcal{H}\ b,$ there is $\mu\in \mathcal{S}_k$ such that $a,b$ have the following form:
\[ a= \begin{pmatrix}
a_1 & a_2 & \dots & a_k\\
c_1 & c_2 & \dots & c_k
\end{pmatrix},\ \ \
b= \begin{pmatrix}
a_1 & a_2 & \dots & a_k\\
c_{1\mu} & c_{2\mu} & \dots & c_{k\mu}
\end{pmatrix}.
\]
Then define $\rho(k,N)$ as follows:
\begin{enumerate}[label=$\bullet$]
\item $a\ \rho(k,N)\ b$ for $a,b$ with $\rk/(a),\rk/(b)<k;$
\item for $\rk/(a)=k=\rk/(b),$ $a\ \rho(k,N)\ b$ if $a\ \mathcal{H}\ b$ and $\mu\in N;$
\item $a\ \rho(k,N)\ a$ for all $a.$ 
\end{enumerate}

\begin{theorem}[see \cite{Liber congs I_n}]
Let $k\leq n$ and $N\trianglelefteq \mathcal{S}_k.$  Then $\rho(k,N)$ is a congruence on $\In/.$  Moreover, every congruence on $\In/$ is of this form.
\end{theorem}

Our objective is to extend the description of congruences on $\In/$ to $\gwi/.$  Notice that we can embed $\In/$ into $\gwi/$ via the map $\eta:\In/\hookrightarrow\gwi/$; where $a\mapsto (1_{aa^{-1}};a).$  If we have $\kappa\subseteq \In/\times \In/$ then we write $\kappa\eta=\{(a\eta,b\eta)\ |\ (a,b)\in\kappa\}.$ 

\begin{lemma}
Let $\kappa$ be an equivalence relation on $\In/,$ and let $\zeta_\kappa$ be the equivalence relation on $\gwi/$ generated by $\kappa\eta.$  Then the restriction of $\zeta_\kappa$ to the $\im/(\eta)$ is $\kappa\eta.$  Moreover, if $\kappa$ is a congruence on $\In/$ and $\zeta_\kappa$ is the congruence on $\gwi/$ generated by $\kappa\eta$ then the restriction of  $\zeta_\kappa$ to $\im/(\eta)$ is $\kappa\eta.$
\end{lemma}

\begin{proof}
We notice that the relation $\xi_\kappa$ on $\gwi/$ defined by 
$$\xi_\kappa=\{ ((g;a),(h;b))\ |\ (a,b)\in \kappa\}$$
is an equivalence relation on $\gwi/.$  Moreover $\kappa\eta\subseteq \xi_\kappa,$ and $\xi_\kappa|_{\im/(\eta)}= \kappa\eta.$  This completes the proof of the first claim. Also when $\kappa$ is a congruence then $\xi_\kappa$ is a congruence.
\end{proof}

Conversely if $\rho$ is a congruence on $\gwi/$ then we can restrict this to a congruence on $\In/$ in two different ways:
\begin{gather*}
\rho_1=\{ (a,b)\ |\ (1_a,a)\ \rho\ (1_b,b)\},\quad 
\rho_2=\{ (a,b)\ |\ \exists g,h\in (G^0)^n \text{ with } (g;a)\ \rho\ (h;b)\}.  
\end{gather*}
Though both are congruences on $\In/,$ in general these are not equal. However it is immediate that $\rho_1\subseteq \rho_2.$

For each ideal $I_m$ of $\gwi/$ we write
$$I_m^\star=\{ ((g;a),(h,b))\mid (g;a),(h,b)\in I_m\}\cup \iota$$
for the Rees congruence on $\gwi/.$  For a congruence $\rho$ on $\gwi/$ it is clear that there is a largest natural number $0\leq k\leq n$ such that $I_k^\star\subseteq \rho,$ which is referred to as the {\em rank} of the congruence. The relation $\rho$ then induces a non-universal congruence on the principal factor $I_{k+1}/I_k.$ Conversely if we are given $\sigma$ a relation on $I_{m}/I_{m-1}^\star$ we can define a relation on $D_m\times D_m:$
$$\overline{\sigma}= \{(a,b)\in D_{m}\times D_{m}\ |\ (a/I^\star_{m-1},b/I^\star_{m-1})\in \sigma \}.$$
Either by reproducing the usual treatment of congruences on $\In/$ adapted for $\gwi/$ or by appealing to general results of Lima \cite{Lima thesis} we can describe congruences on $\gwi/$ as follows.

\begin{theorem}[{\cite[Theorem~3.2.6]{Lima thesis}}]\label{lima}
Let $\chi$ be an idempotent separating congruence, $0\leq m\leq n$ and $\sigma$ be a non universal congruence on $I_{m}/I_{m-1}^\star$ such that $\chi\cap (D_{m}\times D_{m})\subseteq \overline{\sigma}.$  Then
$$\rho(m,\sigma,\chi)=I_{m-1}^\star  \cup \overline{\sigma} \cup\chi$$
is a congruence on $\gwi/.$

Conversely, if $\rho$ is a congruence on $\gwi/$ then with $\chi=\rho\cap\mu,$ $m=\rk/(\rho)+1$ and $\sigma$ chosen such that $\overline{\sigma}=\rho\cap (D_{m}\times D_{m}),$ then $\rho=\rho(m,\sigma,\chi).$
\end{theorem}

Consequently, the problem of describing congruences on $\gwi/$ reduces to describing idempotent separating congruences and to describing congruences on the principal factors.  

\begin{remark}
Since the principal factors are Brandt semigroups all non-universal congruences are idempotent separating congruences so it is possible to formulate Theorem \ref{lima} as the decomposition 
$$\rho=I_{m-1}^\star \cup \overline{\zeta}$$
where $\overline{\zeta}$ is the lift of $\zeta$ - a congruence on $(\gwi/) /I_{m-1}^\star$ - to $\gwi/.$  However for each $m$ a congruence on $(\gwi/) /I_{m-1}^\star$ can be decomposed into a congruence on $I_{m}/I_{m-1}^\star$ and the projection of an idempotent separating congruence on $\gwi/$ onto $(\gwi/) /I_{m-1}^\star.$  Thus it is better to go straight for the decomposition given in Theorem \ref{lima}.
\end{remark}

In order to refine the description of congruences on $\gwi/$ from Theorem \ref{lima} we shall frequently appeal to the correspondence between idempotent separating congruences and normal subsemigroups of $\gwi/$ contained in $E\zeta$ detailed in Theorem \ref{prelim idem sep thm}. Define the function 
$$\Omega\colon E\zeta\rightarrow \bigcup_{1\leq m\leq n} G^m$$
to be the map that ignores zero entries and the final - $\In/$ - coordinate.  It is clear from the description of the $E\zeta$ from Section $2$ that the restriction of $\Omega$ to $E\zeta \cap H_e$ is a isomorphism so $E\zeta \cap H_e \cong G^m.$  

Given $h\in G^m$ write $\overline{{}^e h}$ for the unique element of $(G\cup\{0\})^n$ that has $\bar{{}^e h}_i=0$ for $i\notin \dom/(e)$ and $\overline{{}^eh}\Omega=h$. Let $T\subseteq E\zeta$ be a normal subsemigroup of $\gwi/$ and let $e\in E(\In/)$ be an idempotent with $m=\rk/(e)$. Write $T_e$ for $T\cap H_e=\{(g;e)\in T\}.$

\begin{definition}
A subgroup $K\leq G^m$ is (permutation) invariant if for all $\sigma\in \mathcal{S}_m$ we have that
$$(g_1,g_2,\dots,g_m)\in K\iff (g_{1\sigma},g_{2\sigma},\dots,g_{m\sigma})\in K.$$
Write $PI(G,m)$ for the lattice of invariant subgroups of $G^m.$
\end{definition}

\begin{lemma}\label{T normal inv}
Let $T\subseteq E\zeta$ be a normal subsemigroup of $\gwi/.$  For $e,f\in \In/$ with $\rk/(e)=\rk/(f)=m$ we have that $T_e\cong T_f.$  Moreover the group $T_e\Omega\trianglelefteq G^m$ is normal and invariant, and $T_e\Omega=T_f\Omega.$
\end{lemma}

\begin{proof}
Let $e,f\in E(\In/)$ with $\rk/(e)=\rk/(f).$  Since this implies that $e\ \mathcal{D}\ f$ we may choose $a\in \In/$ such that $aa^{-1}=f$ and $a^{-1}a=e.$ As $T$ is normal for each $(g,e)\in T_e$
$$(1_{aa^{-1}};a)(g;e)(1_{a^{-1}a};a^{-1})=(g_a;aea^{-1})=(g_a;f)\in T_f.$$
Furthermore, the function 
$$\Psi_a\colon T_e\rightarrow T_f;\ (g;e)\mapsto (g_a;f)$$
is an isomorphism. Indeed, it is clear that $\Psi_a$ is a homomorphism, and if $\Psi_{a^{-1}}\colon (h;f)\rightarrow (h_{a^{-1}};e)$ then 
$$(g;e)\Psi_a\Psi_{a^{-1}} = (g_a;f)\Psi_{a^{-1}} = ((g_{a})_{a^{-1}};e)=(g;e).$$

For $a\in H_e$ the homomorphism $\Psi_a$ acts as an element $\sigma_a\in\mathcal{S}_{\rk/(e)}$ permuting the coordinates of the non-zero entries in the $G^0$-component. Moreover the map $H_e(\In/)\rightarrow \mathcal{S}_{\rk/(e)}$ defined by $a\mapsto \sigma_a$ is surjective. This exactly says that $T_e$ is invariant.

Furthermore for $e,f\in E(\In/)$ and $a\in \In/$ with $aa^{-1}=f$ and $a^{-1}a=e$ if $g\in T_e\Omega$ then it is clear that $(\overline{{}^eg};e)\Psi_a\Omega\in T_f\Omega$ and is equal to $g$ under a permutation of the coordinates.  Then as $T_f$ is invariant we have that $g\in T_f\Omega.$  By symmetry this implies $T_e\Omega=T_f\Omega.$ 

To see $T_e\Omega$ is normal let $(g;e)\in T_e$ then observe for $h\in G^m$ that 
$$((\overline{{}^eh})g(\overline{{}^eh^{-1}});e)=(\overline{{}^eh};e)(g;e)(\overline{{}^eh^{-1}};e).$$
As $T$ is normal we have that $((\overline{{}^eh})g(\overline{{}^eh^{-1}});e)\in T_e,$ so $(\overline{{}^eh})g(\overline{{}^eh^{-1}})\Omega = hgh^{-1}\in T_e\Omega,$ and $T_e\Omega$ is normal.
\end{proof}

To define normal subsemigroup contained in $E\zeta$ it therefore suffices to describe a set of invariant subgroups $\{T_i\trianglelefteq G^i\mid 1\leq i\leq n\},$ and the normal subsemigroup is then:
$$T=\bigcup_{e\in E(\In/)}\{(\overline{{}^eg};e)\ |\ g\in T_{\rk/(e)}\}.$$
We call $\{T_i\trianglelefteq G^i\mid 1\leq i\leq n\}$ the {\em defining groups for $T$}.

We write $\pi_m\colon \bigcup_{m\leq i\leq n} G^i\rightarrow G^{m-1}$ for the projection onto the first $m-1$ coordinates.  We say a set $\{T_i\leq G^i\ |\ 1\leq i\leq n\}$ is {\em closed} if $T_i\pi_i\subseteq T_{i-1}$ for each $2\leq i\leq n.$  Notice that when $T_i$ is invariant the projection onto any equally sized subset of the coordinates gives the same result.

\begin{proposition}\label{idemsep congs}
Let $T\subseteq E\zeta$ be a normal subsemigroup of $\gwi/$ and let $\{T_i\leq G^i\mid 1\leq i\leq n\}$ be the defining groups for $T.$  Then each $T_i$ is an invariant normal subgroup and $\{T_i\mid 1\leq i\leq n\}$ is closed.

Moreover if $\{T_i\trianglelefteq G^i\mid 1\leq i\leq n\}$ is a closed set of invariant normal subgroups then
$$T=\bigcup_{e\in E(\In/)}\{ (\overline{{}^eg}; e)\in \gwi/\mid g\in T_{\rk/(e)}\}.$$
is a normal subsemigroup, $T\subseteq E\zeta$ and $\{T_i\leq G^i\mid 1\leq i\leq n\}$ are the defining groups for $T.$
\end{proposition}

\begin{proof}
Recall $\Omega\colon E\zeta \rightarrow \bigcup_{1\leq m\leq n} G^m,$ the function that ignores zero entries in the $(G^0)^n$ component.

Suppose that $T\subseteq E\zeta$ is a normal subsemigroup. By Lemma \ref{T normal inv} we have that each $T_i$ is a normal invariant subgroup. If $e\in E(\In/)$ has domain $\{x_1<x_2<\dots<x_r\}$ then let $f$ be the idempotent with domain $\{x_1<\dots<x_{r-1}\}$ then $g\pi_m=(\overline{{}^eg};e)(1_f;f)\Omega,$ so $g\pi_m \in T_{r-1},$ so $\{T_i\mid 1\leq i\leq n\}$ is closed. 

For the converse, suppose that $\{T_i\mid 1\leq i\leq n\}$ is a closed set of invariant normal subgroups. To see that $T$ is a subsemigroup let $(\overline{{}^eg};e),(\overline{{}^fh};f)\in T$ and observe that
$$(\overline{{}^eg};e)(\overline{{}^fh};f) = ((\overline{{}^eg})(\overline{{}^fh})_e;ef) = ((\overline{{}^eg})1_{ef}(\overline{{}^fh})1_{ef};ef)=((\overline{{}^eg})1_{ef};ef)((\overline{{}^fh})1_{ef};ef).$$
As each $T_i$ is invariant (so the projection onto equally sized subsets has the same image) it is clear that $((\overline{{}^eg})1_{ef};ef)\Omega\in T_{\rk/(e)}\pi_{\rk/(ef)},$ and $((\overline{{}^fh})1_{ef};ef)\Omega\in T_{\rk/(f)}\pi_{\rk/(ef)}.$ As the set of subgroups is closed $T_{\rk/(e)}\pi_{\rk/(ef)},\ T_{\rk/(f)}\pi_{\rk/(ef)}\subseteq T_{\rk/(ef)},$ so $((\overline{{}^eg})1_{ef}(\overline{{}^fh})1_{ef};ef)\Omega\in T_{\rk/(ef)}.$  Also 
$$(\overline{{}^eg})1_{ef}(\overline{{}^fh})1_{ef}=\overline{{}^{ef}(((\overline{{}^eg})1_{ef}(\overline{{}^fh})1_{ef};ef)\Omega)},$$
hence $(\overline{{}^eg};e)(\overline{{}^fh};f)\in T$

It is immediate that $T$ is both full and inverse. To see that $T$ is self conjugate we note that $(g;a)\in \gwi/$ decomposes as $(g;a)=(g;aa^{-1})(1_{aa^{-1}};a).$ Then
$$(g;a)(\bar{{}^fh};f)(g;a)^{-1} = (g;aa^{-1})(1_{aa^{-1}};a)(\bar{{}^fh};f)(1_{a^{-1}a};a^{-1})(g^{-1};a^{-1}a).$$
That $T$ is closed under conjugation by elements of the form $(1_{aa^{-1}};a)$ follows from each $T_i$ being invariant, and closure under conjugation by $(g;aa^{-1})$ follows as each $T_i$ is normal.
\end{proof}

We have shown that to define an idempotent separating congruence it is sufficient to give a closed set of invariant groups $\{T_i\trianglelefteq G^i\mid 1\leq i\leq n\},$ and the idempotent separating congruence can then be expressed explicitly as
$$\chi(T_1,T_2,\dots,T_n) = \{((g;a),(h;a))\ |\ a\in \In/,\ \rk/(a)=i, \text{ and } (h^{-1}g)\Omega\in T_i \}.$$
Furthermore the ordering on idempotent separating congruences coincides with the ordering on closed sets of invariant normal subgroups induced by subgroup inclusion in each degree; that is $\chi(T_1,\dots,T_n)\subseteq \chi(K_1,\dots,K_n)$ if and only if $T_i\subseteq K_i$ for each $1\leq i\leq n.$

\begin{corollary}
The maximum idempotent separating congruence on $\gwi/$ is 
$$\chi(G,G^2,\dots,G^n)=\{((g;a)(h;a))\in \gwi/\times\gwi/\mid a\in \In/ \}.$$
\end{corollary}

The next stage is to describe non universal congruences on $I_m/I_{m-1}^*.$  Each principal factor is a Brandt semigroup, so to describe congruences on the principal factors we appeal to work on describing congruences on Brandt semigroups due to Preston \cite{preston1959congruences}.

\begin{theorem}[see \cite{preston1959congruences}]
Let $S$ be a Brandt semigroup.  Then the lattice of non-universal congruences on $S$ is isomorphic to the lattice of normal subgroups of the principal group.
\end{theorem}

The principal group for the principal factor $I_m/I_{m-1}^*$ is the usual group wreath product of $G$ with the symmetric group $\mathcal{S}_m.$  This is the semidirect product $G^m\rtimes \mathcal{S}_m$ under the action of $\mathcal{S}_m$ on the coordinates of $G^m;$ we write this $G\wr\mathcal{S}_m$

\begin{corollary}
Let $1\leq m\leq n,$ then the lattice of non universal congruences on $I_{m}/I_{m-1}$ is isomorphic to the lattice of normal subgroups of $G\wr\mathcal{S}_m.$ 
\end{corollary}

For a normal subgroup $L\trianglelefteq G\wr\mathcal{S}_m$ write $\sigma_L$ for the corresponding congruence on $I_m/I_{m-1}^*.$ For each $e\in E(\gwi/)$ with $\rk/(e)=m$ let $\Psi_e:H_e\rightarrow G\wr\mathcal{S}_m$ be the isomorphism $(g;a)\mapsto (g\Omega,a\gamma)$ where $\Omega$ is the function that ignores zero entries and $\gamma$ is the usual isomorphism $H_e(\In/)\rightarrow \mathcal{S}_m.$  Define $L_e=\{ (g,a)\in H_e\mid (g;a)\Psi_e \in L\},$ 
$$\sigma_L= \{(u,v)\in I_m/I_{m-1}^*\times I_m/I_{m-1}^*\mid u\ \mathcal{H}\ v,\ u^{-1}v\in L_{u^{-1}u}\}.$$
It is easily seen that the requirement in Theorem \ref{lima} that $\chi\cap(D_m\times D_m)\subseteq \overline{\sigma}$ is equivalent to $\{(t,1)\mid t\in T_{m}\}\subseteq L.$  We can now give a refinement of the description of two sided congruences on $\gwi/.$ 

\begin{theorem}\label{cong decomp}
Let $m\leq n,$ $\{T_i\trianglelefteq G^i\ |\ m+1\leq i\leq n \}$ be a closed set of invariant subgroups, and $L\leq G\wr\mathcal{S}_m$ be such that $\{(t,1)\mid t\in T_{m+1}\pi_{m+1}\}\leq L.$  Then 
$$\rho(m,\{T_i\},L)=I_{m-1}^\star\cup \overline{\sigma_L}\cup \chi(\{G,G^2,\dots,T_{m+1}\pi_{m+1},T_{m+1},\dots,T_n)$$
is a congruence on $\gwi/.$

Moreover all congruences on $\gwi/$ are of this form.
\end{theorem}

The explicit form for $\rho=\rho(m,\{T_i\},L)$ is: $(g;a)\ \rho\ (h;b)$ if one of the following:
\begin{enumerate}[label=$\bullet$]
    \item $\rk/(a)<m$ and $\rk/(b)<m,$
    \item $\rk/(a)>m,$ $a=b$ and $(g^{-1}h)\Omega\in T_{\rk/(a)},$
    \item $\rk/(a)=m=\rk/(b),$ $a\ \mathcal{H}(\mathcal{I}_n)\ b$ and $((g^{-1}h)_{a^{-1}};a^{-1}b)\in L_{a^{-1}a}.$
\end{enumerate}
It is also worth remarking on the relation between the ordering on congruences and the description in Theorem \ref{cong decomp}.  Let $\rho_1=\rho(m_1,\{T_i\mid m_1+1\leq i\leq n\},L_1)$ and $\rho_2=\rho(m_2,\{U_i\mid m_2+1\leq i\leq n\},L_2).$  Then $\rho_1\subseteq \rho_2$ if and only if $m_1\leq m_2,$ $T_i\leq U_i$ for each $m_2+1\leq i\leq n$ and if $m_1=m_2$ then $L_1\leq L_2$ or if $m_1<m_2$ then $\{(t,1)\mid t\in T_{m_2}\}\subseteq L_2.$  This ordering allows us easily compute the intersection and join of congruences. 

\begin{corollary}\label{congs intersection and meet}
Let $m_1\geq m_2,$ and let $\rho_1=\rho(m_1,\{Y_i\ |\ m_1+1\leq i\leq n\}, Z_1)$ and $\rho_2=\rho(m_2,\{W_i\ |\ m_2+1\leq i\leq n\}, Z_2)$ be congruences on $\gwi/.$  Then
\begin{gather*}
\rho_1\vee\rho_2=\rho(m_1,\{U_i\ |\ m_1+1\leq i\leq n\}, A),\\
\rho_1\cap\rho_2=\rho(m_2,\{V_i\ |\ m_2+1\leq i\leq n\}, B)
\end{gather*}
where $U_i=Y_i\vee W_i$ with this the join in $PI(G,m),$ $V_i=Y_i\cap W_i$ and - taking the join and intersection in $\mathfrak{N}(G\wr\mathcal{S}_{m_1})$ - if $m_1=m_2$ then $A=Z_1\vee Z_2$ and $B=Z_1\cap Z_2,$ or if $m_1 > m_2$ then $A=Z_1\vee W_{m_1}$ and $B=Z_2\cap \{(g;1)\mid g\in G^{m_2} \}.$
\end{corollary}

A natural question is what is the relation between $\mathfrak{C}(\gwi/)$ and $\mathfrak{C}(\gwip/).$ We define
$$\Theta: \gwi/\rightarrow \gwip/;\ \ (g_1,\dots,g_n; a)\rightarrow (g_1,\dots,g_n,0; a)$$
where in the image we regard $a$ as an element of $\mathcal{I}_{n+1}.$  It is straightforward that this is an embedding. For a relation $\kappa\subseteq \gwi/\times \gwi/$ we write 
$$\kappa\Theta_2=\{((g;a)\Theta,(h;b)\Theta)\ |\ ((g;a),(h;b))\in \kappa \} \subseteq \gwip/\times\gwip/.$$
The following is straightforward consequence of Theorem \ref{cong decomp}.  

\begin{corollary}
Let $\rho(m,\{T_i\ |\ m+1\leq i\leq n\},L)$ be a congruence on $\gwi/.$  Then $\langle \rho\Theta_2\rangle = \rho(m,\{T_i\ |\ m+1\leq i\leq n+1\},Z),$ where $T_{n+1}$ is the trivial group. 

Moreover, for $\rho$ a congruence on $\gwi/$ we have that $\langle \rho\Theta_2\rangle\cap (\gwi/\Theta\times \gwi/\Theta)=\rho\Theta.$
\end{corollary}

This demonstrates that the map $\mathfrak{C}(\gwi/)\rightarrow \mathfrak{C}(\gwip/)$ defined by $\rho\mapsto\rho\Theta_2$ is also an embedding so we may regard the lattice of congruences on $\gwip/$ as an extension of the lattice of congruences on $\gwi/.$

\section{Normal subgroups of direct and semidirect products}

As a congruence on $\gwi/$ is determined by a set of subgroups of $G^i$ for a range of $i,$ and a subgroup of $G\wr\mathcal{S}_m$ it is a sensible next step to develop a picture of what exactly are the sets of these subgroups. In this section we present a finer analysis of the subgroups of $G^m$ and of $G\wr\mathcal{S}_m$ that arise as components in the prior description of congruences on $\gwi/.$ 

\subsection{Invariant normal subgroups of $G^m$}

 The standard starting point in the consideration of subgroups of direct products of groups is Goursat's lemma.

\begin{theorem}[Goursat's Lemma \cite{goursat}]\label{goursats lemma}
Let $A,B$ be groups. Then the subgroups $N\leq A\times B$ are exactly the sets
$$X(A^\prime,B^\prime,C,D,\theta) =\{ (a,b)\in A^\prime\times B^\prime\ |\ aC\theta=bD \},$$
where $C\trianglelefteq A^\prime\leq A,$ $D\trianglelefteq B^\prime\leq B$ and $\theta\colon A^\prime/C\rightarrow B^\prime/D$ is an isomorphism.
\end{theorem} 

The following is an elementary extension of Goursat's lemma which is more applicable to our case.  

\begin{corollary}\label{inv norm 2}
There is a bijective correspondence between invariant normal subgroups of $G^2$ and triples $(A,C,\theta)$ where $A,C\trianglelefteq G,$ $C\leq A$ and $\theta$ is an automorphic involution of $A/C.$  The invariant normal subgroups of $G^2$ are $X(A,A,C,C,\theta)$ for these triples.
\end{corollary}

For larger $m$ or for semidirect as opposed to direct products the picture grows much more complicated. General extensions of Goursat's lemma exist (for example see \cite{general goursats}), however for our purposes these are too generic and it is possible to directly produce a description tailored to our purposes.  For the remainder of this section unless otherwise stated we will take $m\geq 3$ and let $K\trianglelefteq G^m$ be an invariant normal subgroup.

Define the following 
$$\mathbf{H}(K) =\{ (g,h)\in G^2 \ |\ (g,h,1,\dots,1)\in K\},\quad \mathbf{N}(K) =\{ n\in G \ |\ (n,1,\dots,1)\in K\},$$
and note that $\mathbf{N}(K) = \{ n\in \mathbf{M}(K)\ |\ (n,1)\in \mathbf{H}(K)\}.$  With $\pi\colon G^m\rightarrow G$ being the projection onto the first coordinate define
$$\mathbf{L}(K)=K\pi,\qquad \mathbf{M}(K)=\mathbf{H}(K)\pi.$$
Since $K$ is normal in $G^m$ it is clear that $\mathbf{N}(K),$ $\mathbf{M}(K)$ and $\mathbf{L}(K)$ are normal subgroups of $G,$ and as $K$ is invariant $\mathbf{H}(K)$ is an invariant normal subgroup of $G^2.$

\begin{lemma}\label{L/N abelian}
The commutator $[G,L]=\{glg^{-1}l^{-1}\mid g\in G,\ l\in L\}$ has $[G,L]\subseteq N,$ in particular the quotient group $L/N$ is abelian.
\end{lemma}

\begin{proof}
For each $l\in L$ there is some $k\in K$ with $k=(l,k_2,\dots,k_m).$  As $K$ is normal in $G^m$ for $g\in G$ we obtain that $(glg^{-1},k_2,\dots,k_m)\in K.$  It then follows that $(glg^{-1}l^{-1},1,\dots,1)\in K$ and thus $glg^{-1}l^{-1}\in N.$  Thus we have that $N$ contains the commutator $[G,L].$
\end{proof} 

As $\mathbf{H}(K)$ is an invariant normal subgroup of $G^2$ we may apply Corollary \ref{inv norm 2} to get that there is a automorphic involution $\theta$ of $\mathbf{M}(K)/\mathbf{N}(K)$ such that 
$$\mathbf{H}(K)=\{ (g,h)\ |\ g\mathbf{N}(K)\theta=h\mathbf{N}(K) \}.$$
Suppose that $(g,h,1,\dots,1)\in K,$ then as $K$ is invariant also $(1,h^{-1},g^{-1},1,\dots,1)\in K$ (we here use that $m\geq 3$), hence $(g,1,g^{-1},1,\dots,1)\in K.$  Thus $(g,g^{-1})\in \mathbf{H}(K).$  We then have that $(g\mathbf{N}(K))\theta=g^{-1}\mathbf{N}(K),$ noting that the inverse map is an automorphism since $\mathbf{L}(K)/\mathbf{N}(K)$ (and so also $\mathbf{M}(K)/\mathbf{N}(K)$) is abelian. Therefore
$$\mathbf{H}(K)=\{ (g,h)\in \mathbf{M}(K)^2 |\ hg\in \mathbf{N}(K) \},$$
in particular for $g\in \mathbf{M}(K)$ we have that $(g,g^{-1},1,\dots,1)\in K.$

\begin{lemma}\label{same coset lemma}
Let $K\leq G^m$ be an invariant normal subgroup, and let $M=\mathbf{M}(K)$. If $(g_1,\dots,g_m)\in K,$ then
$$g_1M=g_2M=\dots=g_mM.$$
\end{lemma}

\begin{proof}
Suppose that $(g_1,\dots,g_m)\in K$ then since $K$ is invariant we have $(g_2,g_1,g_3,\dots,g_m),$ and thus 
$$(g_2,g_1,g_3,\dots,g_m)^{-1}(g_1,g_2,g_3,\dots,g_m) = (g_2^{-1}g_1,g_1^{-1}g_2,1,\dots,1). $$
Hence $(g_2^{-1}g_1,g_1^{-1}g_2)\in \mathbf{H}(K),$ so $g_1^{-1}g_2\in M,$ or equivalently $g_1M=g_2M.$ Similarly we obtain that each $g_i$ is in the same left coset of $M,$ so $g_1M=\dots=g_mM.$  
\end{proof}

For an invariant normal subgroup $K\trianglelefteq G^m$ define the function 
$$\phi_K: \mathbf{L}(K)\rightarrow \mathbf{L}(K)/\mathbf{N}(K);\quad g\mapsto y\mathbf{N}(K)\ \text{ where }\ (y,g,g,\dots,g)\in K.$$

\begin{lemma}\label{phi lemma}
Let $K\trianglelefteq G^m$ be an invariant normal subgroup, and let $L=\mathbf{L}(K),$ $M=\mathbf{M}(K)$ and $N=\mathbf{N}(K).$ The function $\phi=\phi_K$ defined previously is a homomorphism.  The restriction of $\phi$ to $M$ has the form $g\mapsto g^{1-m} N,$ and if $g\phi=hN$ then $gM=hM.$  Moreover $(k_1,\dots,k_m)\in K$ if and only if $k_1M=\dots=k_mM$ and $k_1\phi=k_1^{2-m}k_2\dots k_mN.$
\end{lemma}

\begin{proof}
Initially recall that for $x\in M$ we have that $(x,x^{-1})\in \mathbf{H}(K),$ thus for $x_2,\dots,x_m\in M$ we have that $(x_2,x_2^{-1},1,\dots,1),$ $(x_3,1,x_3^{-1},1,\dots,1),$ $\dots,$ $(x_m,1,\dots,1,x_m^{-1})$ are elements of $K.$  Hence $(x_2x_3\dots x_m,x_2^{-1},\dots,x_m^{-1})\in K.$ 

We show now that $\phi$ is well defined. For $g\in L$ there is $k=(g,k_2,\dots,k_m)\in K.$  Then by Lemma \ref{same coset lemma} we have that $gM=k_2M=\dots=k_mM$ so there are $x_2,\dots,x_m\in M$ such that $k_i=gx_i.$ Hence 
$$(g,k_2,\dots,k_m)(x_2\dots x_m,x_2^{-1},\dots,x_m^{-1})=(gx_2\dots x_m,g,\dots,g)\in K.$$
Thus $g\phi=gx_2\dots x_mN.$ Also if $(y,g,\dots,g)\in K$ and $(x,g,\dots,g)\in K$ then it is immediate that $yx^{-1}\in N,$ so $\phi$ is well defined.  

Notice that as $L/N$ is abelian (by Lemma \ref{L/N abelian}) $g(g^{-1}k_2)\dots (g^{-1}k_m)N=g^{2-m}k_2\dots k_mN.$ Therefore we have shown that $(g,k_2,\dots,k_m)\in K$ implies that $g\phi=g^{2-m}k_2\dots k_mN.$  Conversely if $gM=k_2M=\dots=k_mM$ and $g\phi=g^{2-m}k_2\dots k_mN=g(g^{-1}k_2)\dots (g^{-1}k_m)N$ then $(g(g^{-1}k_2)\dots (g^{-1}k_m),g,\dots,g)\in K.$ Since $k_i^{-1}g\in M$ it follows that for each $2\leq i\leq m$ we have $(k_i^{-1}g,1,\dots,g^{-1}k_i,\dots,1)\in K,$ and so
\begin{gather*}
(g(g^{-1}k_2)\dots (g^{-1}k_m),g,\dots,g) (k_m^{-1}g,1,\dots,1,g^{-1}k_m)\dots (k_2^{-1}g,g^{-1}k_2,1,\dots,1)\\
= (g,k_2,\dots,k_m)\in K.
\end{gather*}
Therefore we have shown that $(k_1,\dots,k_m)\in K$ if and only if $k_1M=\dots=k_mM$ and $k_1\phi=k_1^{2-m}k_2\dots k_mN.$

It is straightforward that $\phi$ is a homomorphism, since if $(y,g,\dots,g),$ $(x,h,\dots,h)\in K$ then $(yx,gh,\dots,gh)\in K.$  To see that the restriction of $\phi$ to $M$ is as claimed notice that if $x\in M$ then
$$(x,x,1,\dots,1)(x,1,x,1,\dots,1)\dots(x,1,\dots,1,x)=(x^{m-1},x,\dots,x)\in K.$$
That if $g\phi=hN$ then $gM=hM$ follows immediately from Lemma \ref{same coset lemma}.
\end{proof}

Let $N\leq M\leq L$ be normal subgroups of $G.$ We call a homomorphism $\phi\colon L\rightarrow L/N$ a {\em $(LMNm)$-homomorphism} if the restriction of $\phi$ to $M$ has the form $g\mapsto g^{1-m}N$ and if $g\phi=hN$ implies that $gM=hM.$ Notice that if $\phi$ is an $(LMNm)$-homomorphism then $N\leq \ker/ (\phi).$  We shall now demonstrate that invariant normal subgroups of $G^m$ are determined by suitable quadruples $(L,M,N,\phi).$

\begin{definition}
Let $G$ be a group, then $(L,M,N,\phi)$ is an {\em $m$-invariant quadruple} for $G$ if $L,M,N\trianglelefteq G$ with $N\leq M\leq N$ and $[G,L]\subseteq N,$ and $\phi:L\rightarrow L/N$ is an $(LMNm)$-homomorphism.  
\end{definition}

Given an $m$-invariant quadruple $(L,M,N,\phi)$ we define the following subset of $G^m:$
$$\mathbf{K}_m(L,M,N,\phi)=\{ (g_1,\dots,g_m)\in L^m \ |\ g_1M~\textequal~\dots~\textequal~ g_mM,\ g_1\phi~\textequal~ g_1^{2-m}g_2g_3\dots g_m N\}.$$

\begin{proposition}\label{inv subgroup prop}
Let $K\trianglelefteq G^m$ be an invariant normal subgroup, and let $L=\mathbf{L}(K),$ $M=\mathbf{M}(K)$ and $\mathbf{N}(K).$  Then $(L,M,N,\phi_K)$ is an $m$-invariant quadruple. Furthermore $K=\mathbf{K}_m(L,M,N,\phi_K).$
\end{proposition}

\begin{proof}
As noted previously $L,M,N\trianglelefteq G$ and $N\leq M\leq L.$  That $[G,L]\subseteq N$ is exactly Lemma \ref{L/N abelian}, and $\phi_k$ is an $(LMNm)$-homomorphism by Lemma \ref{phi lemma}.

It remains to show that $K=\mathbf{K}_m(L,M,N,\phi_K).$ By Lemma \ref{phi lemma} we have that $(k_1,\dots,k_m)\in K$ if and only if $k_1M=\dots=k_mM$ and $k_1\phi_K=k_1^{2-m}k_2\dots k_mN.$ However this says exactly that $K=\mathbf{K}_m(L,M,N,\phi_K).$
\end{proof}

\begin{theorem}\label{inv norm sbgps thm}
Let $(L,M,N,\phi)$ be an $m$-invariant quadruple for $G$.  Then $K=\mathbf{K}_m(L,M,N,\theta)$ is an invariant normal subgroup of $G^m.$ Moreover $\mathbf{L}(K)=L,$ $\mathbf{M}(K)=M,$ $\mathbf{N}(K)=N$ and $\phi_K=\phi.$ 

Conversely if $K$ is an invariant normal subgroup then $(\mathbf{L}(K),\mathbf{M}(K),\mathbf{N}(K),\phi_K)$ is an $m$-invariant quadruple and $K=\mathbf{K}_m(\mathbf{L}(K),\mathbf{M}(K),\mathbf{N}(K),\phi_K).$
\end{theorem}

\begin{proof}
Let $(L,M,N,\phi)$ be an $m$-invariant quadruple for $G$ and let $K=\mathbf{K}_m(L,M,N,\theta).$ We initially show that $K$ is a subgroup of $G^m.$ Suppose that $(g_1,\dots,g_m),$ $(h_1,\dots,h_m)\in K,$ so that
\begin{gather*}
g_1M=\dots=g_mM,\quad h_1M=\dots=h_mM,\\
g_1\phi=(g_1^{2-m}g_2\dots g_m) N,\quad h_1\phi=(h_1^{2-m} h_2\dots h_m) N.
\end{gather*}
It is immediate that $g_1h_1M=\dots=g_mh_mM,$ and as $\phi$ is a homomorphism and $L/N$ is commutative (by Lemma \ref{L/N abelian}) we have
\begin{align*}
g_1h_1\phi= g_1\phi h_1\phi & = ((g_1^{2-m}g_2\dots g_m) N)(( h_1^{2-m} h_2\dots h_m) N)\\
& = (g_1h_1)^{2-m} (g_2h_2)\dots (g_mh_m)N.
\end{align*}
Hence $(g_1h_1,\dots, g_mh_m)\in K.$  Furthermore if $(g_1,\dots,g_m)\in K$ then $g_1^{-1}M=\dots=g_m^{-1}M,$ and 
$$g_1^{-1}\phi=(g_1\phi)^{-1}= (g_1^{2-m}g_2\dots g_m)^{-1}N = (g_1^{-1})^{2-m}(g_2^{-1}\dots g_m^{-1})N.$$
Thus $(g_1^{-1},\dots,g_m^{-1})\in K$ and so $K$ is a subgroup of $G^m.$

For the invariant property it is sufficient to show that $K$ is closed under transposition of any two coordinates.  Since $L/N$ is commutative it is immediate that 
$$g_1\phi = g_1^{2-m}(g_2\dots g_i\dots g_j\dots g_m)N = g_1^{2-m} (g_2\dots g_j\dots g_i\dots g_m)N$$
hence $K$ is closed under any transposition within the final $(m-1)$-coordinates.  Therefore it remains to show that $K$ is closed under swapping the first two coordinates.  Suppose that $(g,h,k_3,\dots,k_m)\in K$ so $g\phi=(g^{2-m}hk_3\dots k_m)N$ and $gM=hM.$ Then $g^{-1}h\in M$ and as $\phi$ is an $(LMNm)$-homomorphism we have $(g^{-1}h)\phi=(g^{-1}h)^{1-m}N.$ Therefore
$$h\phi~\textequal~ (gg^{-1}h)\phi ~\textequal~ (g\phi)(g^{-1}h\phi) ~\textequal~ (g^{2-m} hk_3\dots k_m N)(g^{-1}h)^{1-m}N ~\textequal~ h^{2-m} gk_3\dots k_m N. $$
Hence $(h,g,k_3,\dots,k_m)\in K,$ and so $K$ is invariant.

We next show that $K$ is normal. As $K$ is invariant is sufficient to show that we can conjugate in the second coordinate.  Suppose that $(k_1,\dots,k_m)\in K$ and $l\in L.$  Then as $L/M$ is commutative (since $L/N$ is commutative and $N\leq M$) and so $lk_2l^{-1}M=k_2M$ so $k_1M=lk_2l^{-1}M=k_3M=\dots=k_mM.$  Furthermore as $L/N$ is commutative
$$k_1\phi= (k_1^{2-m}k_2\dots k_m)N = (k_1^{2-m}(lk_2l^{-1})k_3\dots k_m) N.$$
Thus $(k_1,lk_2l^{-1},k_3,\dots,g_m)\in K$ so $K$ is normal. 

It remains to show that $\mathbf{L}(K)=L,$ $\mathbf{M}(K)=M,$ $\mathbf{N}(K)=N$ and $\phi_K=\phi.$ As previously remarked for and $(LMNm)$-homomorphism $N\leq\ker/(\phi),$ so for $x\in N$ we have $(x,1,\dots,1)\in K,$ hence $N\subseteq \mathbf{N}(K).$ Suppose that $(x,1,\dots,1)\in K,$ so $xM=M$ and $x\phi=x^{2-m}N.$  However as $x\in M$ and $\phi$ is an $(LMNm)$-homomorphism we also have that $x\phi=x^{1-m}N.$  Thus $x^{1-m}N=x^{2-m}N,$ so $xN=N$ and $x\in N.$  Hence $\mathbf{N}(N)=N.$

For $y\in M$ as $\phi$ is an $(LMNm)$-homomorphism we have $y\phi=y^{1-n} = y^{2-n}y^{-1}N.$ Thus $(y,y^{-1},1,\dots,1)\in K$ so $M\subseteq \mathbf{M}(K).$  Suppose that $(x,y,1,\dots,1)\in K.$  Then certainly $xM=yM=M,$ so $\mathbf{M}(K)\subseteq M$ and the two are equal.

By definition $K\subseteq L^m,$ so $\mathbf{L}(K)\subseteq L.$  Conversely for $l\in L$ choose $x\in l\phi$ and consider $(l,l,\dots,l,x).$ It is straightforward that $(l,\dots,l,x)\in K.$  Thus $\mathbf{L}(K)=L.$

Recall that $\phi_K$ is defined $g\mapsto yN,$ where $(y,g,\dots,g)\in K.$ Suppose that $g\phi_K=y$ so $(y,g,\dots,g)\in K,$ then by the definition of $K$ we have that $y\phi=y^{2-m}g^{m-1}N$ and $gM=yM$ so $gy^{-1}\in M.$ As $\phi$ is an $(LMNm)$-homomorphism $(gy^{-1})\phi=(gy^{-1})^{1-m}N=g^{1-m}y^{m-1}N.$ Then
$$g\phi=(gy^{-1})\phi (y\phi)= (g^{1-m}y^{m-1}N)( y^{2-m}g^{m-1}N) = yN.$$ Thus $\phi_K=\phi.$
\end{proof}

As the ordering on invariant normal subgroups induces the ordering on $\mathfrak{C}_{IS}(\gwi/)$ it is worthwhile to remark upon the ordering of these groups.  It is elementary that $\mathbf{K}_m(L_1,M_1,N_1,\phi_1)\subseteq \mathbf{K}_m(L_2,M_2,N_2,\phi_2)$ exactly when $L_1\subseteq L_2,$ $M_1\subseteq M_2,$ $N_1\subseteq N_2$ and $l\phi_1\subseteq l\phi_2$ for all $l\in L_1.$  It is possible to use this ordering to compute the $m$-invariant quadruple for the intersection and meet of invariant normal subgroups, which can then be combined with Corollary \ref{congs intersection and meet} to give a method to compute the intersections and joins of congruences.

\subsection{Normal subgroups of semidirect products}

The next part of the description of congruences on $G\wr\mathcal{S}_m$ uses normal subgroups of $G\wr\mathcal{S}_m.$  As a wreath product is a special type of semidirect product we shall appeal to Usenko's work \cite{Usenko} on describing subgroups of semidirect products.  We will use the convention that $P$ and $H$ are groups and $\phi\colon P\rightarrow \aut/ H$ is an antihomomorphism. For $p\in P$ and $h\in H$ we write $p\phi = \phi_p$ and $h\phi_p =h^p.$ The semidirect product of $P$ and $H$ is then the set of all ordered pairs $\{(h,p)\mid h\in H,\ p\in P\},$ with the operation:
$$(h,p)(g,q) = (hg^p,pq).$$
We denote this group by $H\rtimes_\phi P.$  A subgroup $K\leq H$ is {\em $\phi$-invariant} if for all $k\in K$ and $p\in P,$ $k^p\in K.$ In the following definition the first $2$ conditions are given in \cite{Usenko}.

\begin{definition}\label{def sncr}
Let $H\rtimes_\phi P$ be a semidirect product $J\leq H$ and $Q\leq P$ be subgroups. We say that a function $\psi: Q\rightarrow H$ is a {\em normal crossed $\mathfrak{R}_\phi^J$ homomorphism} and the triple $(J,H,\psi)$ a {\em normal crossed $\mathfrak{R}_\phi^J$ (NCR) triple} if 
\begin{enumerate}[label=(\roman*)]
\item for all $r,q\in Q$ there is $j\in J$ such that $(rq)\psi = j (r\psi) (q\psi)^r;$\label{sncr 1}
\item for all $q\in Q$ and $j\in J$ we have $(q\psi) j^q (q\psi)^{-1}\in J.$\label{sncr 2}
\end{enumerate}
Furthermore we say $\psi$ is a {\em strongly} normal crossed $\mathfrak{R}_\phi^J$ homomorphism and $(J,H,\psi)$ a {\em strongly normal crossed $\mathfrak{R}_\phi^J$ (SNCR) triple} if $Q\trianglelefteq P,$ $J\trianglelefteq H$ is $\phi$-invariant, and
\begin{enumerate}[label=(\roman*),start=3]
\item for all $q\in Q$ and $p\in P$ we have that $((q\psi)J)^p = ((pqp^{-1})\psi) J;$\label{sncr 3}
\item for all $q\in Q,$ and $h\in H$ we have that $(q\psi) h^q (q\psi)^{-1}\in hJ.$\label{sncr 4}
\end{enumerate}
For a SNCR or NCR triple define the set
$$\mathbf{L}(J,Q,\psi)=\{ (j(q\psi),q)\in H\rtimes_\phi P\mid j\in J,\ q\in Q\}.$$
\end{definition}

To simplify notation we identify $J$ with $\{(j,1)\ |\ j\in J\}\leq H\rtimes_\phi P$ and $H$ with $\{(1,h)\ |\ h\in H\}\leq H\rtimes_\phi P.$ Usenko provides the following description of subgroups of $H\rtimes_\phi P.$

\begin{theorem}[see \cite{Usenko}]\label{usenko}
Let $H\rtimes_\phi P$ be a semidirect product and let $(J,Q,\psi)$ be a NCR triple. Then $\mathbf{L}(J,Q,\theta)$ is a subgroup of $H\rtimes_\phi P.$  

Moreover given $L\leq H\rtimes_\phi P$ a subgroup, let $J=\{h\in H\ |\ (h,1)\in T\}$ and $Q=\{q\in Q\ |\ \exists (u,q)\in T \}.$  For each $q\in Q$ choose $ h=q\theta$ such that $(h,q)\in T.$  Then $(J,Q,\psi)$ is an NCR triple and $L=\mathbf{L}(J,Q,\psi).$  
\end{theorem}

In this description of subgroups of $H\rtimes P$ the group $Q$ can be viewed as the projection of the subgroup $L\leq H\rtimes P$ onto the $P$ coordinate, and $J$ is the kernel of that projection. One particular point that is important to note is that while $\mathbf{L}(J_1,Q_1,\psi_1)=\mathbf{L}(J_2,Q_2,\psi_2)$ does give that $J_1=J_2$ and $Q_1=Q_2$ it does not imply that $\psi_1=\psi_2.$ We specialize Theorem \ref{usenko} to normal subgroups.

\begin{corollary}\label{usenko normal}
Let $H\rtimes_\phi P$ be a semidirect product and $(J,Q,\psi)$ a NCR triple.  Then $L=\mathbf{L}(J,Q,\theta)$ is normal in $H\rtimes_\phi P$ if and only if $(J,Q,\psi)$ is a SNCR triple.
\end{corollary}

\begin{proof}
This is straightforward; it is immediate that $Q$ must be normal and $J$ normal and $\phi$-invariant, also\ref{sncr 3} from Definition \ref{def sncr} is equivalent to $L$ being closed under conjugation by elements of the form $(1,p)$ and \ref{sncr 4} to $L$ being closed under conjugation by elements of the form $(h,1).$  Elements of the form $(1,p)$ and $(h,1)$ generate $H\rtimes_\phi P,$ hence $L$ is normal if and only if $(J,Q,\psi)$ is an SNCR triple.
\end{proof}

Let $(J,Q,\psi)$ be a SNCR triple for $H\rtimes_\phi P$. As $J\trianglelefteq H$ we can consider the quotient group $H/J.$ As $J$ is $\phi$ invariant the antihomomorphism $\phi\colon P\rightarrow \aut/H$ induces an antihomomorhism $\tilde{\phi}\colon P\rightarrow\aut/ H/J;$ $p\mapsto[hJ\mapsto (h\phi_p)J].$  We write $(hJ)^p$ for $(hJ)(p\tilde{\phi}).$ Define $\overline{\psi}:Q\rightarrow H/J$ by $q\overline{\psi}=(q\psi) J.$ As $\psi$ is a SNCR homomorphism we have that there is $j\in J$ such that $(qp)\psi = j(q\psi)(p\psi)^q$ and that $(q\psi) h^q (q\psi)^{-1}\in hJ$ for all $h\in H.$ Thus
$$(qp)\overline{\psi}=(qp)\psi J=(q\psi)(p\psi)^q J = (q\psi)(q\psi)^{-1}(p\psi)(q\psi)J=(p\psi)(q\psi)J = (p\overline{\psi})(q\overline{\psi}).$$
Thus $\overline{\psi}$ is an anti-homomorphism. Using antihomomorphisms $Q\rightarrow H/J$ allows us to define unique triples to each normal subgroup of $H\rtimes_\phi P$.  We say that $(J,Q,\xi)$ is a {\em normal subgroup triple for $H\rtimes_\phi P$} if $Q\trianglelefteq P,$ $J\trianglelefteq H$ is $\phi$-invariant and $\xi:Q\rightarrow H/J$ is an antihomomorphism such that:
\begin{enumerate}[label=(\roman*)]
\item for all $q\in Q$ and $p\in P$ we have that $(q\xi)^p = ((pqp^{-1})\xi) ;$\label{triple 1}
\item for all $q\in Q$ and $h\in H$ we have that $ h^q J =(q\xi)^{-1}  hJ (q\xi).$\label{triple 2}
\end{enumerate}
For a normal subgroup triple $(J,Q,\xi)$ we define the set
$$\mathbf{W}(J,Q,\xi)= \{ (x,y)\in H\rtimes_\phi P \mid y\xi =xJ\}.$$

\begin{corollary}\label{dank triple cor}
Let $H\rtimes_\phi P$ be a semidirect product. The normal subgroups of $H\rtimes_\phi P$ are exactly the sets $\mathbf{W}(J,Q,\xi)$ for normal subgroup triples $(J,Q,\xi).$  Moreover normal subgroups of $H\rtimes_\phi P$ are in bijection with normal subgroup triples.
\end{corollary}

\begin{proof}
This follows immediately from Corollary \ref{usenko normal} noting that if $(J,Q,\psi)$ is a SNCR triple then $(J,Q,\overline{\psi})$ is a normal subgroup triple (with $\overline{\psi}$ as defined just prior to this corollary), and $\mathbf{L}(J,Q,\psi)=\mathbf{W}(J,Q,\overline{\psi}).$  That a normal subgroup uniquely determines a normal subgroup triple follows from the elementary observation that if $\mathbf{L}(J,Q,\psi_1)=\mathbf{L}(J,Q,\psi_2)$ then $\overline{\psi_1}=\overline{\psi_2}.$ 
\end{proof}

We now apply this to $G\wr\mathcal{S}_m,$ and unless otherwise stated we continue to assume that $m\geq 3.$  Suppose we have an normal subgroup $K\trianglelefteq G^m$ that is invariant under the action of $\mathcal{S}_m,$ which is exactly saying that $K$ is an invariant subgroup of $G^m$ so there is an $m$-invariant quadruple $(L,M,N,\phi)$ such that $K=\mathbf{K}_m(L,M,N,\phi).$ 

\begin{lemma}
Let $(L,M,N,\phi)$ be an $m$-invariant quadruple for $G$ and let $K=\mathbf{K}_m(L,M,N,\phi);$ let $Q\neq \{1\}$ be a normal subgroup of $\mathcal{S}_m,$ and let $\xi\colon Q\rightarrow G^m/K$ be an anti-homomorphism such that $(K,Q,\xi)$ is a normal subgroup triple for $G\wr\mathcal{S}_m.$ Then $L=M=G,$ so $K=\mathbf{K}(G,G,N,g\mapsto g^{1-m}N).$ 
\end{lemma}
\begin{proof}
Suppose for a contradiction that $M\neq G,$ so there is $x\in G$ with $xM\neq M.$  Take $1\neq a\in Q$ that has $1a=2$ (note that this is possible as $Q$ is non-trivial and all non-trivial normal subgroups of $\mathcal{S}_m$ for $m\geq 3$ contain such elements). As $(K,Q,\xi)$ is a normal subgroup triple, for $g\in G^m$ we have $(g^a)K=(a\xi)^{-1}(gK)(a\xi).$  Choose $(a_1,\dots,a_m)\in a\xi,$ then this implies that 
\begin{align*}
(1,x,1,\dots,1)K=(x,1,\dots,1)^a K & =(a_1,\dots,a_m)^{-1}(x,1,\dots,1)(a_1,\dots,a_m)K\\
& = (a_1^{-1}xa_1,1,\dots,1)K.
\end{align*}
This then gives that 
$$(1,x^{-1},1,\dots,1)(a_1^{-1}xa_1,1\dots,1)=(a_1^{-1}xa_1,x^{-1},1,\dots,1)\in K.$$
Recall that
$$K= \{\{ (g_1,\dots,g_m)\in L^m \mid g_1M=\dots=g_mM,\ g_1N\phi= g_1^{1-m}g_2g_3\dots g_m N\}. $$
This implies that $a_1^{-1}xa_1M=xM=M$, a contradiction, so we must have that $M=G.$ That $L=G$ follows as $M\leq L,$ and the final claim is then immediate. 
\end{proof}

Let $K=\mathbf{K}_m(G,G,N,g\mapsto g^{1-n}N)$ and notice that 
$$K=\{ (g_1,\dots,g_m)\mid g_1g_2\dots g_m\in N\}.$$ 
In this case the function 
$$\Xi:G^m/K\rightarrow G/N;\quad  (g_1,\dots,g_m)K\mapsto g_1\dots g_m N$$
is easily seen to be an isomorphism. The action of $\mathcal{S}_m$ on $G^m$ (permuting the coordinates) carries forward to the quotient group $G^m/K$ and this induces an action of $\mathcal{S}_m$ on $G/N$ via the isomorphism $\Xi.$  As $G/N$ is abelian (by Lemma \ref{L/N abelian}) this induced action is trivial, so with $K^\prime=\{(k;1)\in G\wr\mathcal{S}_m\mid k\in K\}$ we obtain that $(G\wr\mathcal{S}_m)/K^\prime\cong G/N\times \Sn/.$ 

Let $(K,Q,xi)$ be a normal subgroup triple, and recall that $K$ is the kernel of the projection of $\mathbf{W}(K,Q,xi)\leq G\wr\mathcal{S}_m$ onto the final coordinate. By the correspondence theorem normal subgroups of $G\wr\mathcal{S}_m$ that correspond to a normal subgroup triple $(K,Q,\xi)$ for a fixed $K$ (we vary the $Q$ and $\xi$) are the lifts of normal subgroups of $(G\wr\mathcal{S}_m)/K$ to $G\wr\mathcal{S}_m$ (where by the lift of a subgroup $C\leq A/B$ to $A$ we mean $a\in A$ such that $aB\in C$) such that the projection of the subgroup onto the second coordinate has trivial kernel. 

We can use Goursat's lemma (Theorem \ref{inv norm 2}) to obtain the set of normal subgroups $R\trianglelefteq G/N\times \mathcal{S}_m$ in terms of subgroups $A,B\trianglelefteq G/N$ and $U,V\trianglelefteq \mathcal{S}_m$ and isomorphisms $\psi\colon U/V\rightarrow A/B.$ As $B$ is the kernel of the projection onto the $\mathcal{S}_m$ coordinate we may assume that $B$ is trivial, so may simplify this to a pair $Q,$ $\zeta$ where $Q\trianglelefteq \mathcal{S}_m$ and $\zeta\colon Q\rightarrow G/N$ is a homomorphism.  The subgroup of $G/N\times \mathcal{S}_m$ is then
$$\{( q\theta, q)\ |\ q\in Q\}\trianglelefteq G/N\times \Sn/,$$
We can now proceed to summarize this in the following description of subgroups of $G\wr\mathcal{S}_m.$ 

\begin{theorem}\label{subgroups of gws}
Let $G$ be a group and $m\geq 3$ an integer.  The following is a complete list of all normal subgroups of $G\wr\mathcal{S}_m$:
\begin{enumerate}[label={(\roman*)}]
\item for each $K\trianglelefteq G^m$ an invariant normal subgroup:
$$\{(k,1)\mid k\in K\}\trianglelefteq G\wr\mathcal{S}_m;$$
\item for each $N\trianglelefteq G$ with $G/N$ abelian, $Q\trianglelefteq \mathcal{S}_m$ non trivial and $\xi\colon Q\rightarrow G/N$ a homomorphism:
$$ \{ (g,q)\mid q\in G,\ (gK)\Xi= q\xi\}\trianglelefteq G\wr\mathcal{S}_m;$$
where $K=\mathbf{K}_m(G,G,N,g\mapsto g^{1-m}N)$ and $\Xi\colon G^m/K\rightarrow G/N$ is defined $(g_1,\dots,g_m)K\mapsto g_1\dots g_m N. $
\end{enumerate}
\end{theorem}

When $m=2$ reproducing similar arguments gives that normal subgroups $L\trianglelefteq G\wr\mathcal{S}_2$ are of the form either: $L=\{(k,1)\mid k\in K\}$ for $K\trianglelefteq G^2$ an invariant normal subgroup, or when $K\trianglelefteq G^2$ is an invariant normal subdirect product so $K=X(G,G,N,N,\theta)$ as in Corollary \ref{inv norm 2}, then $L=\{(k,s)\mid (kN)\alpha=s\zeta\}$ where $\zeta\colon \mathcal{S}_2\rightarrow G/N$ is a homomorphism and $\alpha\colon G^2/K\rightarrow G/N$ is defined $(g,h)K\mapsto (gN)((hN)\theta)^{-1}.$

\section{The Number of Congruences}

We shall now delve deeper into the consideration of the set of congruences on $\gwi/$ and will provide an answer to the question: what is the asymptotic growth of $|\mathfrak{C}(\gwi/)|$?  We recall that for $\In/$ the number of congruences grows linearly in $n,$ and the number of normal subgroups of $G^n$ grows exponentially in $n.$  We shall show that for $\gwi/$ the growth is polynomial.

\begin{proposition}\label{finitesubs}
Let $G$ be a finite group. Then there is an integer $\lambda_1(G)$ such that for each $m\in \mathbb{N}$ the number of permutation invariant subgroups of $G^m$ is less than $\lambda_1(G).$
\end{proposition}

\begin{proof}
Notice that as for each $m$ the group $G^m$ is finite it follows that there are only finitely many subgroups of $G^m.$  Therefore it suffices to prove the claim for $m$ sufficiently large, which in this case is at least $3.$  Let
$$Z=\{(L,M,N,\theta)\ |\ N\leq M\leq L,\ N,M,L\trianglelefteq G,\ \phi:L\rightarrow L/N\}.$$
By Theorem \ref{inv norm sbgps thm} we have that invariant normal subgroups of $G^m$ are determined by $m$-invariant quadruples. For each $m$ we have that $\iqm/$ (the set of $m$-invariant quadruples) is a subset of $Z.$ Thus to show that the number of invariant normal subgroups is bounded we shall show that $Z$ is finite.

We note first that if $G$ is finite then the lattice of normal subgroups of $G$ is finite.  Thus there are finitely many chains of subgroups of length $3,$ so there are finitely many possibilities for $L,M,N.$ For each choice of $L,M,N$ since $L$ and $L/N$ are finite there are finitely many homomorphisms $L\rightarrow L/N.$   Therefore $Z$ is finite so for each $m$ as $|\iqm/|\leq |Z|,$ the the number of $m$-invariant subgroups of $G^m$ is bounded.
\end{proof}

In general for a finite group $G$ it is difficult to calculate precise values or even efficient bounds for $\lambda_1(G),$ and is similarly hard to compute precise values for $|\iqm/|.$ A significant factor in both these calculations is the number of normal subgroups of $G.$ We can extend Proposition \ref{finitesubs} to normal subgroups of $G\wr\mathcal{S}_m.$

\begin{corollary}\label{finitesubs2}
Let $G$ be a finite group, then there is finite number $\lambda_2(G)$ such that for all $m$ the number of normal subgroups of $G\wr\mathcal{S}_m$ is at most $\lambda_2(G).$
\end{corollary}

\begin{proof}
As $G\wr\mathcal{S}_m$ is finite for each $m$ it again suffices to prove the result for $m$ sufficiently large, here at least $5.$ By Theorem \ref{subgroups of gws} all subgroups of $G\wr\mathcal{S}_m$ are of one of two types, either $\{(k,1)\mid k\in K\}$ for $K$ an invariant normal subgroup of $G^m,$ or $ \{ (g,q)\mid q\in G,\ (gK)\Xi= q\xi\}$ where $K=\mathbf{K}(G,G,N,x\mapsto x^{-n}),$ $Q\trianglelefteq \mathcal{S}_m$ is non trivial, $\xi\colon Q\rightarrow G/N$ is a homomorphism and $\Xi\colon G^m/K\rightarrow G/N$ is defined $(g_1,\dots,g_m)K\mapsto (g_1\dots g_m)N.$

By Lemma \ref{finitesubs} there are at most $\lambda_1(G)$ invariant normal subgroups of $G^m$ so there at most $\lambda_1(G)$ normal subgroups of $G\wr\mathcal{S}_m$ of the first type.  Also, there are finitely many normal subgroups $N\trianglelefteq G$ which have $G/N$ abelian. When $G/N$ is abelian and $Q\trianglelefteq \mathcal{S}_m$ is non trivial, a homomorphism $\xi: Q\rightarrow G/N$ is the trivial map if $Q=\mathcal{A}_m,$ and if $Q=\mathcal{S}_m$ then there is precisely one homomorphism for each element of $G/N$ of order at most $2$ (here we assume without loss of generality that $m\geq 5$ so that $\mathcal{A}_m$ is simple). Thus for each $N\trianglelefteq G$ the number of homomorphisms $\xi:Q\rightarrow G/N$ with $Q\trianglelefteq \Sn/$ non trivial is exactly $r_N + 2$ where $r_N$ is the number of elements of order $2$ in $G/N,$ therefore there are at most $r_N+2$ normal subgroups of $\mathcal{S}_m$ arising from this $N.$

We have shown that there are at most 
$$\lambda_2(G)=\lambda_1(G)+ \sum_{N\trianglelefteq G,\ [G,G]\subseteq N} (r_N+2)$$
normal subgroups of $G\wr\mathcal{S}_m$ for each $m$ and, since there are finitely many normal subgroups of $G,$ $\lambda_2(G)$ is finite.
\end{proof}

The following is a standard elementary combinatorial result, we state it here as we shall refer to it frequently.

\begin{lemma}\label{increasing chains}
Let $C$ be an increasing chain of length $c,$ and let $v_k$ be the number of sequences $t_1\leq t_2\leq\dots\leq t_k$ of length $k$ where each $t_i \in C.$ Then
$$v_k=\binom{k+c-1}{c-1}.$$
Moreover there are $A,B\in \mathbb{N}$ such that for all $k$
$$Ak^{c-1}\leq v_k\leq Bk^{c-1}$$
\end{lemma}

By Proposition \ref{idemsep congs} idempotent separating congruences correspond to closed sets of invariant normal subgroups, and by Theorem \ref{inv norm sbgps thm} each invariant normal subgroup is of the form 
$$\mathbf{K}_m(L,M,N,\phi)=\{ (g_1,\dots,g_m)\in L^m \ |\ g_1M~\textequal~\dots~\textequal~g_mM,\ g_1\phi~\textequal~ g_1^{2-m}g_2g_3\dots g_m N\}$$
for an $m$-invariant quadruple $(L,M,N,\phi).$ The projection onto the first $(m-1)$ coordinates (though any choice of $m-1$ coordinates is equivalent) is the set
$$\{(g_1,\dots,g_{m-1})\in L^{m-1}\ |\ g_1M=\dots=g_{m-1}M\}= \mathbf{K}_{m-1}(L,M,M,x\mapsto xM).$$
The following lemma is an immediate consequence.

\begin{lemma}
Let $G$ be a group, and $\{ \mathbf{K}_i(L_i,M_i,N_i,\phi_i)\trianglelefteq G^i\mid 1\leq i\leq n\}$ a closed set of invariant normal subgroups.  Then $L_i\subseteq L_{i-1}$ and $M_i\subseteq N_{i-1}$ for each $2\leq i\leq n.$
\end{lemma}

In this way, to each idempotent separating congruence on $\gwi/$ we associate a set of normal subgroups $\{L_i,M_i,N_i\ |\ i=1,\dots,n\}$ of $G$ which is ordered as shown in Figure \ref{lattsubs}, where the arrows denote subset inclusion.  We will refer to a lattice arising in this way from a congruence as a {\em (congruence) induced lattice}, and say that the congruence {\em induces} the lattice.

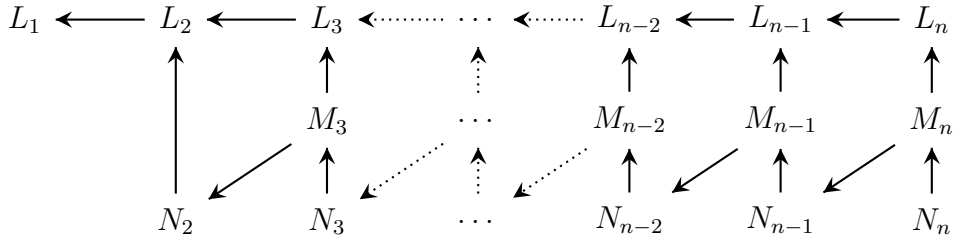
\begin{figure}[!ht]
\centering

\begin{tikzpicture}[scale=.67,-,auto,thick,main node/.style={square},normal node/.style={inner sep=5pt,minimum size=20pt}]

\node[normal node] (a1) at (15,0){$N_{n}$};
\node[normal node] (a2) at (15,2){$M_{n}$};
\node[normal node] (a3) at (15,4){$L_{n}$};

\node[normal node] (b1) at (12,0){$N_{n-1}$};
\node[normal node] (b2) at (12,2){$M_{n-1}$};
\node[normal node] (b3) at (12,4){$L_{n-1}$};

\node[normal node] (c1) at (9,0){$N_{n-2}$};
\node[normal node] (c2) at (9,2){$M_{n-2}$};
\node[normal node] (c3) at (9,4){$L_{n-2}$};

\node[normal node] (d1) at (6,0){$\dots$};
\node[normal node] (d2) at (6,2){$\dots$};
\node[normal node] (d3) at (6,4){$\dots$};

\node[normal node] (e1) at (3,0){$N_{3}$};
\node[normal node] (e2) at (3,2){$M_{3}$};
\node[normal node] (e3) at (3,4){$L_{3}$};

\node[normal node] (f1) at (0,0){$N_{2}$};
\node[normal node] (f3) at (0,4){$L_{2}$};

\node[normal node] (g3) at (-3,4){$L_{1}$};

\tikzset{myptr/.style={decoration={markings,mark=at position 1 with %
    {\arrow[scale=1.5,>=stealth]{>}}},postaction={decorate}}}
\tikzset{dots/.style={dotted,decoration={markings,mark=at position 1 with %
    {\arrow[scale=1.5,>=stealth]{>}}},postaction={decorate}}}
    
\draw [myptr] (a1) -- (a2);
\draw [myptr] (a2) -- (a3);
\draw [myptr] (a3) -- (b3);
\draw [myptr] (a2) -- (b1);

\draw [myptr] (b1) -- (b2);
\draw [myptr] (b2) -- (b3);
\draw [myptr] (b3) -- (c3);
\draw [myptr] (b2) -- (c1);

\draw [myptr] (c1) -- (c2);
\draw [myptr] (c2) -- (c3);
\draw [dots] (c3) -- (d3);
\draw [dots] (c2) -- (d1);

\draw [dots] (d1) -- (d2);
\draw [dots] (d2) -- (d3);
\draw [dots] (d3) -- (e3);
\draw [dots] (d2) -- (e1);

\draw [myptr] (e1) -- (e2);
\draw [myptr] (e2) -- (e3);
\draw [myptr] (e3) -- (f3);
\draw [myptr] (e2) -- (f1);

\draw [myptr] (f1) -- (f3);
\draw [myptr] (f3) -- (g3);

\end{tikzpicture}
\caption{\small Lattice of normal subgroups of $G$ induced by a congruence on $\gwi/$}
\label{lattsubs}
\end{figure}

When $G$ is finite by Lemma \ref{finitesubs} for each $m$ there are at most $\lambda_1(G)$ invariant normal subgroups of $G^m.$  As an idempotent separating congruence is determined by $n$ invariant normal subgroups (one for each $1\leq i\leq n$) this implies that $|\mathfrak{C}_{IS}(\gwi/)|\leq \lambda_1(G)^n.$  However we can significantly improve on this bound for large $n.$  For a group $G,$ a maximal strictly increasing chain of normal subgroups is called a {\em chief series,} and the maximum length of a chief series is the {\em chief length} which we write $c(G).$

\begin{proposition}\label{idemsize}
Let $G$ be a finite group with $c(G)=c.$ Then there are $A,B\in \mathbb{N}$ such that $$An^{c-1}\leq |\mathfrak{C}_{IS}(\gwi/)|\leq Bn^{2(c-1)}.$$
\end{proposition}

\begin{proof}
This result concerns asymptotic behaviour of $|\mathfrak{C}(\gwi/)|$ so we may assume that $n$ is much larger than $c.$ First we demonstrate the lower bound. Notice that for $L\trianglelefteq G$ the group $\mathbf{K}_i(L,L,L,l\mapsto L)=L^i$ is an invariant normal subgroup of $G^i$ for each $i.$  Thus for each sequence of normal subgroups $L_n\leq L_{n-1}\leq\dots\leq L_1$ the set $\{L_i^i\mid 1\leq i\leq n\}$  is a closed set of invariant normal subgroups, so $\chi(L_1,L_2^2,\dots,L_n^n)$ is an idempotent separating congruence. Different sequences of normal subgroups of $G$ give different congruences.  By Lemma \ref{increasing chains} there is some $A$ such there are at least $An^{c-1}$ sequences of normal subgroups, so we have that $An^{c-1}\leq |\mathfrak{C}_{IS}(\gwi/)|.$
 
In order to prove the upper bound we first show that for an induced lattice $Y,$ there is an upper bound to the number of idempotent separating congruences which induce $Y.$  Let $\{L_i,M_i,N_i\mid 3\leq i\leq n\}\cup \{N_2,L_2,L_1\}$ be the labels of the vertices in $Y,$ so $N_n\leq M_n\leq N_{n-1}\leq \dots\leq N_3\leq M_3$ is a sequence of normal subgroups.  As $c(G)=c$ there are at most $c-1$ values of $i$ for which $N_i\neq M_i.$ When $N_i=M_i$ if $\mathbf{K}_i(L_i,N_i,N_i,\phi)$ is an invariant normal subgroup then $\phi:L_i\rightarrow L_i/N_i$ is the standard quotient homomorphism, so when $M_i=N_i$ there is precisely one invariant normal subgroup $K\trianglelefteq G^i$ with $\mathbf{L}(K)=L_i,$ $\mathbf{M}(K)=M_i$ and $\mathbf{N}(K)=N_i$  

Let $q$ be the largest number of homomorphisms $L\rightarrow L/N$ where we vary $L$ and $N$ over normal subgroups of $G.$ By Proposition \ref{finitesubs} there are fewer than $\lambda_1(G)$ invariant normal subgroups $K\trianglelefteq G^2,$ so at most $\lambda_1(G)$ invariant normal subgroups that have $N_2=\{g\in G\mid (g,1)\in K\}.$ Hence there are at most $\lambda_1(G)q^{c-1}$ idempotent separating congruences which induce $Y.$  Thus it suffices to show that there are at most $B^\prime n^{c-1}$ induced lattices.

To this end notice that a congruence induced lattice (with ordering as shown in Figure \ref{lattsubs}) can be decomposed into two sequences.  The first $L_n\leq L_{n-1}\leq\dots\leq L_2\leq L_1$ of length $n,$ and the second $N_n\leq M_n\leq N_{n-1}\leq M_{n-1}\leq\dots\leq M_3\leq N_2$ of length $2n-3.$  When we ignore repeats in these sequences the resulting chains are each subchains of chief series.

Since $G$ is a finite group there are finitely many chief series; say that there are $r$ chief series, then there are $r^2$ pairs of chief series.  On the other hand by Lemma \ref{increasing chains} the number of sequences of length $k$ arising from a chain of length $x$ is bounded above by $Dk^{x-1}$ for some $D\in \mathbb{N}.$  Thus given a pair chief series (which each have maximum length $c$), there are fewer than $(D(2n-3)^{c-1})(Dn^{c-1})$ distinct pairs of sequences, of lengths $n$ and $2n-3$ respectively, which reduce to subchains of this pair of chief series when repeats are ignored.  It follows that there are at most $r^2(D(2n-3)^{c-1})(Dn^{c-1})$ congruence induced lattices. Hence we have that
$$|\mathfrak{C}_{IS}(\gwi/)|\leq \lambda_1(G)q^{c-1}r^2(D(2n-3)^{c-1})(Dn^{c-1}) \leq (\lambda_1(G)q^{c-1}D^2r^22^{c-1})n^{2(c-1)},$$
and $(\lambda_1(G)q^{c-1}D^2r^22^{c-1})=B$ is a constant determined by $G$.  This completes the proof of the result. 
\end{proof}

This demonstrates that number of idempotent separating congruences on $\gwi/$ for a finite group is related to the chief length of $G$. Up to order of the polynomial these bounds are in general the best possible, there are groups with arbitrarily large chief length that attain either the maximum or the minimum order growth for $|\mathfrak{C}_{IS}(\gwi/)|$ from Proposition \ref{idemsize}. 

\begin{example}\label{example max}
For the maximum growth of $|\mathfrak{C}(\gwi/)|$ we consider the group $G=\mathbb{Z}_2^x,$ which has chief length $c(G)=x.$  For each $X\leq Y\leq G$ and for each $i,$ $\mathbf{K}_i(X,Y,Y,x\mapsto xY)\trianglelefteq G^i$ is an invariant normal subgroup.  Let $C$ be a chief series of length $x,$ then by Lemma \ref{increasing chains} there are at least $A^\prime(n/2)^{x-1}=An^{x-1}$ sequences of subgroups of length $n/2$ which reduce to a subchain of $C$ when repeats are ignored.  For a pair of decreasing sequences of subgroups $\{W_i\mid 1\leq i\leq n/2\}$ and $\{Y_i\mid 1\leq i \leq n/2\},$ define $K_i=\mathbf{K}_i(G,W_i,W_i,g\mapsto gW_i)$ and $K_{n/2+i}=\mathbf{K}_{n/2+i}(Y_i,\{1\},\{1\},y\mapsto y)$ for $1\leq i\leq n/2.$ Then $\{K_i\mid 1\leq i\leq n\}$ is a closed set of invariant normal subgroups so defines an idempotent separating congruence. Moreover different choices of $\{W_i\mid 1\leq i\leq n/2\}$ and $\{Y_i\mid 1\leq i \leq n/2\}$ give distinct congruences.  Thus $|\mathfrak{C}(\gwi/)|\geq (An^{x-1})^2=A^2n^{2(x-1)}.$
\end{example}

\begin{example}\label{example min}
For a group that attains the minimum growth for $|\mathfrak{C}_{IS}(\gwi/)|$ we consider the group $G=\mathcal{A}_5^x.$  We note that if $J\trianglelefteq G$ then $J=J_1\times\dots\times J_x$ for $J_i\trianglelefteq\mathcal{A}_5.$ Moreover if also $N\trianglelefteq G$ with $J/N$ abelian then $J=N.$  Thus the only invariant normal subgroups of $G^m$ are of the form $\mathbf{K}_m(L,L,L,l\mapsto L).$  Hence idempotent separating congruences on $\gwi/$ exactly correspond to chains of normal subgroups of $G$ of length $n,$ of which by Lemma \ref{increasing chains} there are at most $Bn^{x-1}.$
\end{example}

Since the chief length plays an important role in the size of $|\mathfrak{C}(\gwi/)|$ is worth noting that in general it is not possible to do better than the trivial bound on the chief length; that is that $c(G)$ is at most the number of prime factors (counted with multiplicity) of $|G|.$ 

\begin{theorem}
Let $G$ be a finite group $c(G)=c.$  Then there are $A,B\in \mathbb{N}$ such that
$$An^{c}\leq |\mathfrak{C}(\gwi/)|\leq Bn^{2c-1}.$$
\end{theorem}

\begin{proof}
The upper bound is straightforward, we note that by Theorem \ref{cong decomp} each congruence $\rho$ decomposes in terms of a Rees congruence, a normal subgroup of $G\wr\mathcal{S}_m$ and an idempotent separating congruence. There are $n+1$ ideals, by Lemma \ref{finitesubs2} at most $\lambda_2(G)$ normal subgroups of $G\wr\mathcal{S}_m$ and by Proposition \ref{idemsize}, at most $B^\prime n^{2(c-1)}$ idempotent separating congruences.  Thus
$$|\mathfrak{C}(\gwi/)|\leq \lambda_2(G)B^\prime (n+1) n^{2(c-1)}= Bn^{2c-1}.$$ 

We now prove the lower bound. Let $\{C_i\mid i\in I\}$ be the set of decreasing sequences of normal subgroups of $G$ of length $n/2.$ We have seen that $|I|\geq A^\prime (n/2)^{c-1}.$ For $C_i=\{L_1\supseteq L_2\supseteq\dots\supseteq L_{n/2}\}$ consider the set $$X_i=\{G,G^2,\dots,G^{n/2},K_1,K_2\dots,K_{n/2}\}$$
where $K_j=\mathbf{K}_{n/2+j}(L_j,L_j,L_j,l\mapsto L_j)=L_j^{n/2+j}$ for $1\leq j\leq n/2.$  This is a closed set of invariant normal subgroups, and for distinct sequences of normal subgroups of $G$ the corresponding idempotent separating congruences are different.  Thus $\rho_{X_i,m}=\rho(m,X_i,\{(g;1)\mid g\in G^m\})$ is a distinct congruence for each $i\in I$ and $1\leq m\leq n/2.$ Thus 
$$|\mathfrak{C}(\gwi/)|\geq A^\prime (n/2)^{c-1}(n/2) = An^c.$$
\end{proof}

As is the case for $\mathfrak{C}_{IS}(\gwi/)$ in general these are the best possible bounds. There are groups that attain the maximum and minimum polynomial growth for the size of the congruence lattice; the groups considered in Examples \ref{example max} and \ref{example min} again attain the maximum and minimum respectively.

\section*{Acknowledgements}
This project forms part of the work toward my PhD at the University of York, supported by EPSRC grant EP/N509802/1.  I would like to thank my supervisor, Professor Victoria Gould, for all her help and guidance during this project.

\begin{small}

\end{small}

\end{document}